\documentclass[12pt, a4paper, reqno]{amsart}
\usepackage{xcolor}
\usepackage[normalem]{ulem}
\usepackage[hdivide={2.5cm,,2.5cm}, vdivide={3.5cm,,2.8cm}]{geometry}
\usepackage{amssymb}
\usepackage{amsthm}
\usepackage{amsmath}
\usepackage{mathtools}

\newcommand{\REM}[1]{\relax}

\numberwithin{equation}{section}

\newcommand{\de}{\partial}
\newcommand{\Hol}{{\sf Hol}}

\def\Re{{\sf Re}\,}
\def\Im{{\sf Im}\,}

\newcommand{\UH}{\mathbb{H}}

\newcommand{\Real}{\mathbb{R}}
\newcommand{\Natural}{\mathbb{N}}
\newcommand{\Complex}{\mathbb{C}}
\newcommand{\ComplexE}{\overline{\mathbb{C}}}

\newcommand{\UD}{\mathbb{D}}
\newcommand{\clS}{\mathcal{S}}

\newcommand{\Maponto}%{\stackrel{%\scriptscriptstyle
{\xrightarrow{\scriptstyle \!\mathsf{onto}\,}}

\newcommand{\Mapinto}
{\xrightarrow{\hbox{\lower.2ex\hbox{$\scriptstyle \smash{\mathsf{into}}$}}\,}}

\let\R=\Real
\let\D=\UD

\let\C=\Complex

\newcommand{\proofbox}{\hfill$\Box$}

\newcommand{\STOP}{\par\hbox to\textwidth{\color{red}\leaders\hbox{\,STOP\,}\hfil}\par}

\newcommand{\mcite}[1]{\csname b@#1\endcsname}

\theoremstyle{theorem}

\def\dist{{\rm dist}}

\def\id{{\sf id}}

\def\Re{{\sf Re}\,}
\def\Im{{\sf Im}\,}

\newtheorem{theorem}{Theorem}[section]
\newtheorem{lemma}[theorem]{Lemma}
\newtheorem{proposition}[theorem]{Proposition}
\newtheorem{corollary}[theorem]{Corollary}

\theoremstyle{definition}
\newtheorem{definition}[theorem]{Definition}
\newtheorem{example}[theorem]{Example}

\theoremstyle{remark}
\newtheorem{remark}[theorem]{Remark}
\numberwithin{equation}{section}

\def\mydot#1{\smash{\stackrel{\,\lower.12ex\hbox{\text{\LARGE.}}}{#1}}\vphantom{\raise.2ex\hbox{$#1$}}}
\newcommand{\anglim}{\angle\lim}

\newcommand{\SHOWCORRECTIONS}{%
\newcommand{\nv}[1]{{\color{green!60!black}##1}}%
\newcommand{\comm}[1]{{\color[rgb]{0.5,0,0.5}##1}}%
\newcommand{\dv}[1]{{\color[rgb]{0.65,0.74,0.79}\sout{##1}}}%
\newcommand{\IN}[1]{{\color[rgb]{1.00,0.33,0.33}##1}}
\newcommand{\ID}[1]{{\color[rgb]{0.65,0.55,0.62}\sout{##1}}}%
\newcommand{\IC}[1]{{\color[rgb]{0.00,0.00,1}##1}}%
\newcommand{\RD}[1]{\textcolor{red}{##1}}%
}%

\newcommand{\HIDECORRECTIONS}{%
\newcommand{\nv}[1]{##1}
\newcommand{\dv}[1]{\relax}%
\newcommand{\comm}[1]{\relax}%
\let\IN=\nv%
\let\ID=\dv%
\let\IC=\comm%
\newcommand{\RD}[1]{##1}%
}%

\SHOWCORRECTIONS

        \let\bibciteOLD=\bibcite
        \renewcommand{\bibcite}[2]{\bibciteOLD{#1}{\textsc{#2}}}
        \makeatletter
        \renewcommand{\@biblabel}[1]{[\textsc{#1}]\hfill}
        \makeatother

\title[]{Chordal Loewner chains with quasiconformal extensions}
\author[P. Gumenyuk]{Pavel Gumenyuk\,$^\dag$}
\address{Institutt for matematikk og naturvitenskap, Universitetet i Stavanger, 4036 Stavanger, Norway}
\email{pavel.gumenyuk@uis.no}
\author[I. Hotta]{Ikkei Hotta\,$^\ddag$}
\address{Department of Applied Science, Yamaguchi University 2-16-1 Tokiwadai, Ube 755-8611, Japan}
\email{ihotta@yamaguchi-u.ac.jp}
\thanks{$^\dag$ Partially supported by the FIRB grant Futuro in Ricerca ``Geometria Differenziale Complessa e Dinamica Olomorfa'' n. RBFR08B2HY}
\thanks{$^\ddag$ Supported by JSPS KAKENHI Grant Numbers 13J02250, 26800053.}
\date{\today}

\begin{document}

\begin{abstract}
In 1972, Becker [J. Reine Angew. Math. \textbf{255}
  (1972), 23--43] discovered a construction of quasiconformal extensions making use of the classical radial Loewner chains. In this paper we develop a chordal analogue of Becker's construction. As an application, we establish new sufficient conditions for quasiconformal extendibility of holomorphic functions and give a simplified proof of one well-known result by Becker and Pommerenke for functions in the half-plane [J. Reine Angew. Math. {\bf 354} (1984), 74--94].
\end{abstract}

\maketitle

%\tableofcontents

    %   ++++++++++++++++++++++++++++++++++++++++++++++++++++++
    %
    %       Section 1
            \section{Introduction}
    %
    %   ++++++++++++++++++++++++++++++++++++++++++++++++++++++

Loewner Theory, which goes back to the parametric representation of univalent functions introduced by Loewner~\cite{Loewner:1923} in 1923, has recently undergone significant development in various directions,  including Schramm's stochastic version of the Loewner differential equation~\cite{Schramm:2000} and the new intrinsic approach suggested by Bracci, Contreras and D\'\i{}az-Madrigal~\cite{BracciCD:evolutionI,BracciCD:evolutionII}.

This paper is devoted to one classical application of Loewner Theory, namely to sufficient conditions for quasiconformal extendibility of holomorphic functions.
It is well-known that many sufficient conditions for univalence of holomorphic functions~$f$ in the unit disk $\UD:=\{z\in\Complex\colon|z|<1\}$ can be proved by constructing a one-parameter family of holomorphic functions~$(f_t)_{t\ge0}$ with ${f_0=f}$ and showing that $(f_t)$ is a classical radial Loewner chain (see Sect.\,\ref{SS_classical-radial}). For the latter part, one typically uses Pommerenke's criterion of a Loewner chain~\cite[Theorem~6.2]{Pom:1975}, which essentially consists of two conditions: (a)~$\Re p>0$, where $p:=\mydot{f_t}/(zf'_t)$ with $\mydot{f_t}$ standing for the derivative w.r.t. the real parameter~$t$ and $f'_t$ denoting the derivative w.r.t. the complex variable~$z$;  (b) a certain condition on the growth of~$f_t$ as $t\to+\infty$. See \cite[Chapter\,6]{Pom:1975} for several important examples.

Quasiconformal mappings appear both as a classical object of study and as a powerful tool in modern Complex Analysis. In particular, normalized holomorphic functions in~$\UD$ extendible to quasiconformal automorphisms of~$\Complex$ play a very important role  in Teichm\"uller Theory, because they can be identified with the elements  of the universal Teichm\"uller space, see, e.g., \cite{Takhtajan}.
Criteria for quasiconformal extendibility of holomorphic functions have been the main topic of numerous studies, see, e.g., \cite{AksentevS:2002, Krushkal:Handbook, Becker:1980, Milne:1991, SugawaPol, Sugawa, Krzyz:1976, Schober:1975} and references therein.

In 1972, Becker discovered an unexpected connection between classical radial Loewner chains and quasiconformal extendibility. He found an additional condition on the function~$p$ (see Theorem~\ref{TH_Beckerthm}) which ensures that all the elements of the Loewner chain~$(f_t)$ have q.c.-extensions to~$\C$. This result, together with Pommerenke's criterion mentioned above, allows one to obtain various sufficient conditions for quasiconformal extendibility of a holomorphic function in~$\UD$. It is worth to mention that almost all classical sufficient conditions can be deduced in this way, see, e.g., \cite{Becker:1980}.

In this paper we extend Becker's construction to the so-called Loewner chains \textit{of chordal type}, see Definition~\ref{D_LCh_chordal-type}. This kind of Loewner chains represents more complicated but, at the same time, more flexible object in comparison with the radial Loewner chains. This allows us to enrich the topic with various ideas coming, in particular, from the theory of one-parameter semigroups of holomorphic self-maps.

The paper is organized as follows. We start with the Preliminaries (Section~\ref{S_preliminaries}) giving some necessary background, in particular,  from Loewner Theory and from the theory of quasiconformal mappings.

In Section~\ref{S_qc-ext-of-Loewner_chains-and-evol_fam} we will prove our analogue of Becker's result.

\begin{theorem}
    %   ---------------------------------------------------
    %
    %       Main theorem 1
            \label{TH_main-thrm-in-short}
    %
    %   ---------------------------------------------------
Let $(f_t)$ be a Loewner chain of choral type in~$\UH:=\{z\colon\Re z>0\}$ with associated Herglotz function~$p(z,t):=-\mydot{f_t}(z)/f'_t(z)$. Let $k\in[0,1)$. If for all $z\in\UH$ and a.e.\,$t\ge0$,
\begin{equation}\label{U(k)}
    p(z,t) \in U(k):=\left\{ w \in \C \colon \left|\frac{w-1}{w+1}\right| \leq k\right\},
\end{equation}
then:
\begin{itemize}
\item[(i)]   for any $t\ge0$, $f_t$ has a $k$-quasiconformal  extension to~$\ComplexE$ with a fixed point at~$\infty$;
\item[(ii)] for any~$s\ge0$ and any~$t\ge s$, $\varphi_{s,t}:=f_s^{-1}\circ f_t$  has a $k$-quasiconformal extension to~$\ComplexE$ with a fixed point at~$\infty$.
\end{itemize}
%   ---------------------------------------------------
    %
    % END OF     Main theorem 1
     %
    %   ---------------------------------------------------
\end{theorem}
\begin{remark}
Although formally our result is very similar to that of Becker, there is an essential difference. Becker's construction, as well as its generalization found by Betker~\cite{Betker:1992}, produces quasiconformal maps with two fixed points: an interior fixed point at~$0\in\UD$  and an exterior fixed point  at~$\infty\not\in\overline{\UD}$,--- while the extensions we obtain in Theorem~\ref{TH_main-thrm-in-short} have a boundary fixed point at~$\infty\in\partial\UH$.
\end{remark}

Theorem~\ref{TH_main-thrm-in-short} follows immediately from more technical Theorem~\ref{TH_alldetails}, which we prove in Section~\ref{SS_maintheorem}. Using automorphisms of~$\UH$ we extend this result to the case of unbounded Herglotz functions~$p$ in Section~\ref{SS_case-of-unbdd}, see Corollary~\ref{C_qc-ext}.

Unfortunately, Pommerenke's criterion of a Loewner chain does not apply to the chordal setting. In Section~\ref{S_LCh-criteria} we establish several partial analogues of Pommerenke's criterion for Loewner chains of chordal type.

In Section~\ref{S_sufficient-conditions}, we use these results to obtain several concrete sufficient conditions for quasiconformal extendibility of holomorphic functions in~$\UH$ and $\UD$, see Theorems~\ref{TH_sufficient-qc1},\,\,\ref{TH_sufficient-qc2},\,\,\ref{TH_QCvia-derivative} and Corollaries~\ref{C_qc-sufficient}\,\,and\,\,\ref{CR_QCvia-derivative}. Up to our best knowledge, these results are all new except for Theorem~\ref{TH_sufficient-qc1}, which follows from~\cite[Theorem\,4.1]{SugawaPol}. We also give a simplified proof of a well-known criterion due to Becker and Pommerenke~\cite[Satz\,2]{BecPom:1984}.

To conclude the introduction, let us mention one interesting question.
In view of the fact that the same condition \eqref{U(k)} appears both in Becker's result (Theorem~\ref{TH_Beckerthm}) for radial Loewner chains and in our main result (Theorem \ref{TH_main-thrm-in-short}) for Loewner chains of chordal type, it would be natural to ask whether for any univalent holomorphic function $f:\UD\to\Complex$, $f(0)=f'(0)-1=0$, admitting a $k$-q.c. extension to~$\C$, there exists a (classical radial) Loewner chain $(f_t)$ with $f_0=f$ such that the corresponding classical Herglotz function~$p$ satisfies this condition~\eqref{U(k)}. The answer, even in special cases, does not seem to be known. Closely related but somewhat different problems were  studied in~\cite{Vasiliev1, Vasiliev}.

\section{Preliminaries}
\label{S_preliminaries}
\subsection{Loewner Theory}

Most of the applications make use of the classical Loewner Theory based mainly on the works of  Loewner~\cite{Loewner:1923}, Kufarev~\cite{Kufarev:1943}, Pommerenke~\cite{Pom:1965}, \cite[Chapter\,6]{Pom:1975} and Kufarev \textit{et al}~\cite{Kufarev:1968}. In our case, however, it turns out to be much more beneficial to work in the framework of a general and, in a certain sense, more intrinsic approach suggested recently by Bracci, Contreras and D\'\i{}az-Madrigal \cite{BracciCD:evolutionI,BracciCD:evolutionII}, see also~\cite{MR2789373, ABHK}, which can be regarded as a non-autonomous extension of the theory of one-parameter semigroups in the unit disk.

Denote by $\Hol(A,B)$ the set of holomorphic maps from $A$ to $B$.

\begin{definition}\label{DF_one-param_semigroup}
A \emph{one-parameter semigroup (in the unit disk)} is a continuous semigroup homomorphism~$t\mapsto\phi_t$ from the semigroup $\big([0,+\infty),\cdot+\cdot\big)$ endowed with the Euclidian topology to the semigroup $\Hol(\UD,\UD)$ endowed with the topology of locally uniform convergence in~$\UD$.
\end{definition}
Equivalently, one can think of a one-parameter semigroup as of a family~$(\phi_t)_{t\ge0}\subset\Hol(\UD,\UD)$ such that
\begin{itemize}
\item[(i)] $\phi_0=\id_\UD$;
\item[(ii)] $\phi_s\circ\phi_t=\phi_t\circ\phi_s=\phi_{t+s}$ for any $s,\,t\ge0$;
\item[(iii)] $\phi_t(z)\to z$ as $t\to0^+$ for any~$z\in\UD$.
\end{itemize}

It is known, see~\cite{Berkson-Porta}, that there is a one-to-one correspondence between the one-parameter semigroups in~$\UD$ and the class of holomorphic functions~$G\colon\UD\to\Complex$, the so-called \emph{infinitesimal generators}, that can be represented by the Berkson\,--\,Porta formula
\begin{equation}\label{EQ_BP_f-la}
G(z)=(\tau-z)(1-\overline \tau z)p(z),
\end{equation} where $\tau\in\overline\UD$ and $p\in\Hol(\UD,\Complex)$, $\Re p\ge0$.
Namely, any infinitesimal generator $G$ defines a one-parameter semigroup~$(\phi_t)$ by means the initial value problem
\begin{equation}\label{EQ_INI-one-param}
\frac{d\phi_t(z)}{dt}=G(\phi_t(z)),\quad t\ge0,\qquad \phi_0(z)=z,
\end{equation}
which has a unique solution $[0,+\infty)\ni t\mapsto \phi_t(z)\in\UD$ for each~$z\in\UD$. Conversely, given a one-parameter semigroup~$(\phi_t)$, there exists a unique infinitesimal generator~$G$ such that~\eqref{EQ_INI-one-param} holds.

\begin{remark}
If $p\not\equiv0$, then the point~$\tau$ in~\eqref{EQ_BP_f-la}, called the \emph{Denjoy\,--\,Wolff point} of the semigroup~$(\phi_t)$ (in short, the \emph{DW-point}), is the common Denjoy\,--\,Wolff point (see Definition~\ref{DF_DW-pint}) of all $\phi_t$'s that are different from~$\id_\UD$.
\end{remark}

It is also known, see, e.g., \cite[Theorems\,1.4.22,\,1.4.23]{Abate:book}, that every non-trivial\footnote{A one-parameter semigroup $(\phi_t)$ is called \emph{trivial} if $\phi_t=\id_\UD$ for all~$t\ge0$.} one-parameter semigroup can be linearized by means of a conformal change of variable. If the DW-point~$\tau$ belongs to~$\UD$, then there exists a unique univalent holomorphic function ${h:\UD\to\Complex}$, $h(\tau)=0$, $h'(\tau)=1$, and a number $\lambda$, $\Re \lambda \le 0$, that satisfy the \emph{Schr\"oder functional equation} $h\circ \phi_t=e^{\lambda t} h$ for all~$t\ge0$. If~$\tau\in\partial\UD$, then there exists a univalent holomorphic function~$h:\UD\to\Complex$, $h(0)=0$, that solves the \emph{Abel functional equation} ${h\circ \phi_t=h+t}$ for all~$t\ge0$.
\begin{definition}
The function~$h$ above is called the \emph{K\oe nigs function} of~$(\phi_t)$.
\end{definition}

One of the three fundamental notions in Loewner Theory (which corresponds to the K{\oe}nigs function in the theory of one-parameter semigroups, see, e.g.~\cite[Sect.\,5]{MR2789373}) is defined as follows.
\begin{definition}[\cite{MR2789373}]\label{DF_L-ch}
A family $(f_t)_{t\ge0}$ of holomorphic functions~$f_t:\UD\to\Complex$ is said to be a {\sl Loewner chain (in the unit disk)} if the following three conditions hold:
\begin{enumerate}
\item[LC1.] each function $f_t:\D\to\C$ is univalent;

\item[LC2.] $f_s(\D)\subset f_t(\D)$ for all $0\leq s < t<+\infty$;

\item[LC3.] for any compact set $K\subset\mathbb{D}$  there exists a
non-negative locally integrable function $k_{K}:[0,+\infty)\to\Real$ such that
\[
|f_s(z)-f_t(z)|\leq\int_{s}^{t}k_{K}(\xi)d\xi
\]
for all $z\in K$ and all $s,t\ge0$ with $t\ge s$.
\end{enumerate}
\end{definition}

For any Loewner chain $(f_t)$,  the functions~$\varphi_{s,t}:=f_t^{-1}\circ f_s\in\Hol(\UD,\UD)$, $t\ge s\ge 0$, form a two-parameter family~$(\varphi_{s,t})$ satisfying the following three conditions, see \cite[Theorem~1.3]{MR2789373},
\begin{enumerate}
\item[EF1.] $\varphi_{s,s}=\id_{\mathbb{D}}$;

\item[EF2.] $\varphi_{s,t}=\varphi_{u,t}\circ\varphi_{s,u}$ whenever $0\leq
s\leq u\leq t<+\infty$;

\item[EF3.] for each $z\in\mathbb{D}$ there exists a
locally integrable function ${k_{z,T}\colon[0,+\infty)\to[0,+\infty)}$ such that
\[
|\varphi_{s,u}(z)-\varphi_{s,t}(z)|\leq\int_{u}^{t}k_{z,T}(\xi)d\xi
\]
whenever $0\leq s\leq u\leq t<+\infty.$
\end{enumerate}

\begin{definition}[\cite{BracciCD:evolutionI}]\label{DF_EF}
A family~$(\varphi_{s,t})_{t\ge s\ge0}\subset\Hol(\UD,\UD)$ satisfying the above conditions~EF1, EF2, and~EF3 is called an \textsl{evolution family (in the unit disk)}.
\end{definition}
\begin{remark}
In~\cite{BracciCD:evolutionI} and \cite{MR2789373}, the definitions of Loewner chains and evolution families contain an integrability order  parameter~$d\in[1,+\infty]$. For our purposes this parameter is irrelevant, so we work with the most general case of order~$d=1$.
\end{remark}

Evolution families play the role of a non-autonomous analogue of one-parameter semigroups in~$\UD$. Accordingly, evolution families in the unit disk appear to be non-autonomous flows of the \emph{time-variable} infinitesimal generators, the so-called Herglotz vector fields.

\begin{definition}[\cite{BracciCD:evolutionI}]
\label{DF_HerglVF} A function $G:\mathbb{D}\times[0,+\infty)\to \mathbb{C}$ is called a {\sl Herglotz
 vector field (in the unit disk)} if it satisfies the following three conditions:
\begin{enumerate}
\item[H1.] for any $z\in\mathbb{D},$ the function $[0,+\infty)\ni t\mapsto G(z,t)$ is measurable;

\item[H2.] for any compact set $K\subset\mathbb{D}$  there
exists a non-negative locally integrable function $k_{K}\colon[0,+\infty)\to\Real$ such that $|G(z,t)|\leq k_{K}(t)$
for all $z\in K$ and for a.e.~$t\in[0,+\infty)$;
\item[H3.] for a.e. $t\in [0,+\infty)$,
$G_t:=G(\cdot, t)$ is an infinitesimal generator.
\end{enumerate}
\end{definition}
For any Herglotz vector field~$G$ there exists a unique evolution family~$(\varphi_{s,t})$ such that for each $s\ge0$ and each $z\in\UD$ the function $s\le t\mapsto w(t):=\varphi_{s,t}(z)$ solves the following initial value problem for the \textsl{general version of the Loewner\,--\,Kufarev ODE},
\begin{equation}\label{EQ_LK-ODE-prelim}
\frac{dw(t)}{dt}=G\big(w(t),t\big),\quad w(0)=z,
\end{equation}
which is to be understood as a Carath\'eodory first-order ODE, see, e.g., \cite[Ch.\,18]{Kurzweil} or \cite[Sect.\,2]{CDGannulusI}.
Conversely, every evolution family~$(\varphi_{s,t})$ is obtained in this way and the corresponding Herglotz vector field~$G$ in~\eqref{EQ_LK-ODE-prelim} is unique up to a set of measure zero on the~$t$-axis, see~\cite[Theorem~1.1]{BracciCD:evolutionI}.

Finally, given an evolution family~$(\varphi_{s,t})$, there exists a Loewner chain~$(f_t)$ such that $\varphi_{s,t}=f_t^{-1}\circ f_s$ whenever $0\le s\le t$. This Loewner chain~$(f_t)$, which we say to be {\sl associated with}~$(\varphi_{s,t})$, is unique up to the post-composition with conformal maps of~$\Omega:=\cup_{t\ge0} f_t(\UD)$, see~\cite[Theorem~1.7]{MR2789373}. Substituting $\varphi_{s,t}=f_t^{-1}\circ f_s$ to~\eqref{EQ_LK-ODE-prelim}, one obtains the differential equation for~$(f_t)$, i.e. the \textsl{general version of the Loewner\,--\,Kufarev PDE}
\begin{equation}\label{EQ_LK-PDE-prelim}
\frac{\partial {f_t}(z)}{\partial t}=-f_t'(z)G(z,t),
\end{equation}
where $G$ is the Herglotz vector field corresponding to~$(\varphi_{s,t})$.

It is a trivial but important remark that all the definitions and results mentioned above in this section, with just a few obvious modification imposed by the change of variable, can be repeated in case of the unit disk~$\UD$  replaced by any domain $D\subset\C$ conformally equivalent to~$\UD$.  In particular, we will make use of one-parameter semigroups, Loewner chains, evolution families, and Herglotz vector fields in the right half-plane $D=\UH:=\{z\colon\Re z>0\}$.

To reduce the amount of Loewner chains associated to a given evolution family, we will consider, in certain cases, what we call \textsl{range-normalized Loewner chains}.
\begin{definition}\label{DF_range-norm}
A Loewner chain~$(f_t)$ in the right half-plane~$\UH$ is said to be \textsl{range-normalized} if $\Omega:=\cup_{t\ge0} f_t(\UH)$ is either $\C$ or a half-plane.
\end{definition}
Similarly, a Loewner chain in~$\UD$  is said to be {range-normalized} if $\Omega:=\cup_{t\ge0} f_t(\UD)$ is either $\C$ or a (euclidian) disk.
\begin{remark}
The so-called standard Loewner chains in~$\UD$ defined in~\cite{MR2789373} are range-normalized, but not vice versa.  Standard Loewner chains are in one-to-one correspondence with evolution families in~$\UD$, see \cite[p.\,981]{MR2789373}. In this paper, where mainly the case of the half-plane~$\UH$ is considered, we  prefer to work with the weaker normalization given by Definition~\ref{DF_range-norm}, which still permits Loewner chains associated with the same evolution family differ by a M\"obius transformation.
\end{remark}

%   +++++++++++++++++++++++++++++++++++++
    %
    %
       \subsection{Classical radial Loewner chains}\label{SS_classical-radial}
%
    % +++++++++++++++++++++++++++++++++++++
In modern literature, by the \textsl{classical radial} Loewner chains and evolution families one means the special kind of Loewner chains and evolution families defined by Pommerenke\footnote{These definitions and basic facts mentioned below can be found in~\cite[Chapter~6]{Pom:1975}.}, which can be obtained from the general Definitions~\ref{DF_L-ch} and~\ref{DF_EF} by imposing some additional normalizations, namely,
\begin{equation}\label{EQ_classical-radial-normalizations}
\varphi_{s,t}(0)=0,\quad \varphi_{s,t}'(0)=e^{s-t},\quad\text{and}\quad f_t(0)=0,\quad f'_t(0)=e^{t}\quad\text{ whenever $t\ge s\ge0$.}
\end{equation}

Similarly to the general case, there is a correspondence between classical  radial evolution families~$(\varphi_{s,t})$ and classical  radial Loewner chains~$(f_t)$. Moreover, it is one-to-one: $(f_t)$ can be reconstructed from its evolution family by means of the formula~$f_s=\lim_{t\to+\infty}e^t\varphi_{s,t}$ for all~$s\ge0$. Normalization~\eqref{EQ_classical-radial-normalizations} obviously forces elements of evolution families to share the DW-point at~$z=0$. The corresponding Herglotz vector fields take then the following special form
$$
G(w,t)=-w p(w,t),
$$
where $p$ is a \textsl{classical Herglotz function}, i.e. a function $p:\UD\times[0,+\infty)\to\Complex$ such that $p(z,\cdot)$ is measurable for any fixed~$z\in\UD$, $p(\cdot,t)$ is holomorphic in~$\UD$ with~$\Re p>0$  and satisfies the normalization~$p(0,t)=1$ for a.e. $t\ge0$.

Thus, any classical radial Loewner chain  satisfies the classical version of equation~\eqref{EQ_LK-PDE-prelim}
\begin{equation}\label{EQ_LK-PDE}
\frac{\partial f_t(z)}{\partial t}=zf_t'(z)p(z,t).
\end{equation}

%   +++++++++++++++++++++++++++++++++++++
    %
    %
   \subsection{Loewner chains of chordal type}
    %
    %   +++++++++++++++++++++++++++++++++++++
 According to the classical Denjoy\,--\,Wolff Theorem, for any holomorphic self-map $\varphi\colon \UD\to\UD$ different from~$\id_\UD$ there exists a unique point~$\tau$ in the closure~$\overline\UD$ of~$\UD$ such that $\varphi(\tau)=\tau$ and $|\varphi'(\tau)|\le 1$. In case $\tau\in\partial\UD$, this should be understood in the sense of angular limits: $\anglim_{z\to\tau}\varphi(z)=\tau$ and $\varphi'(\tau):=\anglim_{z\to\tau}(\varphi(z)-\tau)/(z-\tau)\le 1$. This point plays a special role in the theory of holomorphic self-maps, because the iterates $\varphi^{\circ n}(z)$ converge to~$\tau$ locally uniformly in~$\UD$ as $n\to+\infty$, unless $\varphi$ is an elliptic automorphism,  see, e.g., \cite[\S1.2.2,\,\S1.3.2]{Abate:book}.

\begin{definition}\label{DF_DW-pint}
The point~$\tau=\tau(\varphi)$ defined above is called the \textsl{Denjoy\,--\,Wolff point} (the \textsl{DW-point}, in short) of~$\varphi$.
\end{definition}

\begin{remark}
Passing from~$\UD$ to~$\UH$ by means of the Cayley transform $H(z):={(1+z)/(1-z)}$, $z\in\UD$, we can define the DW-point also for holomorphic self-maps of~$\UH$.
In particular, if~$\varphi\in\Hol(\UH,\UH)\setminus\{\id_\UH\}$ and $\tau(\varphi)=\infty$, then there exists the so-called Carath\'eodory angular derivative of~$\varphi$ at~$\infty$, $$\varphi'(\infty):=\anglim_{z\to\infty}\frac{\varphi(z)}z=\frac1{(H^{-1}\circ\varphi\circ H)'(1)}\ge 1,$$ and by the half-plane version of the Julia\,--\,Wolff\,--\,Carath\'eodory Theorem (see, e.g., \cite[Ch.\,IV~\S26]{Valiron}), \begin{equation}\label{EQ_JWC}\Re \varphi(z)\ge \varphi'(\infty)\,\Re z\quad\text{for all~$z\in\UH$.}\end{equation}
\end{remark}

In the classical radial case, all elements of an evolution family different from~$\id$ share the same interior DW-point.
Following~\cite{CDGgeometry}, in this paper we mainly consider the case of the boundary DW-point at~$\infty$.
\begin{definition}\label{D_LCh_chordal-type}
An evolution family $(\varphi_{s,t})$ in~$\UH$ is said to be \textsl{of chordal type}, if it has the common DW-point at~$\infty$, i.e. $\tau(\varphi_{s,t})=\infty$ whenever $t\ge s\ge 0$ and $\varphi_{s,t}\neq\id_\UH$.
Accordingly, a Loewner chain~$(f_t)$ in~$\UH$ is said to be \textsl{of chordal type} if the corresponding evolution family~$(\varphi_{s,t})=(f_t^{-1}\circ f_s)$ is of chordal type.
\end{definition}

For the infinitesimal generators of non-trivial one-parameter semigroups $(\phi_t)$ in~$\UH$ with the common DW-point at~$\infty$, the Berkson\,--\,Porta formula~\eqref{EQ_BP_f-la} takes the form ${G(z)=p(z})$, $z\in\UH$, where $p\in\Hol(\UH,\C)\setminus\{0\}$, $\Re p\ge0$.
This explains the result below.  We make use of the following definition.
\begin{definition}
A function $p:\UH\times[0,+\infty)\to\Complex$ satisfying the following conditions
\begin{itemize}
\item[HF1.] $p(\cdot,t)$ is holomorphic in~$\UH$ for a.e. $t\ge0$;
\item[HF2.] $p(z,\cdot)$ is locally integrable on~$[0,+\infty)$ for all~$z\in\UH$;
\item[HF3.] $\Re p\ge0$,
\end{itemize}
is called a \textsl{Herglotz function} in~$\UH$.
\end{definition}
\begin{proposition}[\protect{\cite[Theorem 6.7, Corollary 7.2]{BracciCD:evolutionI}}]\label{PR_VF-chordal-type}
Let $(\varphi_{s,t})$ be an evolution family in~$\UH$ and $G$ be the corresponding Herglotz vector field in~$\UH$. Then $(\varphi_{s,t})$ is of chordal type if and only if $G(z,t)=p(z,t)$ for all $z\in\UH$ and  all ${t\in[0,+\infty)\setminus N}$, where $N\subset[0,+\infty)$ is a set of measure zero and $p$ is a Herglotz function in~$\UH$.
\end{proposition}

In view of Proposition~\ref{PR_VF-chordal-type}, any Loewner chain~$(f_t)$ of chordal type and the corresponding evolution family $(\varphi_{s,t})$ satisfy the chordal Loewner\,--\,Kufarev PDE and ODE, respectively,
\begin{equation}\label{EQ_ODE-chordal-type}
\frac{d\varphi_{s,t}(z)}{dt}=p\big(\varphi_{s,t}(z),t\big),\quad t\ge s,\quad \varphi_{s,s}(z)=z,
\end{equation}
\begin{equation}\label{EQ_PDE-chordal-type}
\frac{\partial f_t(z)}{\partial t}=-p(z,t)f_t'(z),\quad t\ge0,~~z\in\UH,
\end{equation}
with some Herglotz function~$p$ determined uniquely up to a set of measure zero.
\begin{definition}
We call~$p$ in~\eqref{EQ_ODE-chordal-type} and~\eqref{EQ_PDE-chordal-type} the \textsl{Herglotz function associated with} $(\varphi_{s,t})$ and  $(f_t)$. Conversely, given a Herglotz function $p:\UH\to\C$, the unique evolution family~$(\varphi_{s,t})$ in~$\UH$ satisfying~\eqref{EQ_ODE-chordal-type} and all the Loewner chains of chordal type corresponding to~$(\varphi_{s,t})$ will be called the \textsl{chordal evolution family} and the \textsl{chordal Loewner chains associated with} the Herglotz function~$p$.
\end{definition}

\begin{remark}
Essentially, one special case of Loewner chains of chordal type was studied by Kufarev \textit{et al} \cite{Kufarev:1968}, Aleksandrov~\cite{AleksST}, Aleksandrov \textit{et al}~\cite{Aleks1983}, Goryainov and Ba~\cite{GoryainovBa:1992} and by Bauer~\cite{Bauer:2005}, see also \cite[Sect.\,5]{CDGgeometry}. However, it is worth to mention that the Loewner chains considered by these authors do not satisfy the hypothesis of our main Theorem~\ref{TH_main-thrm-in-short}. It might be an interesting problem to find a construction of q.c.-extensions making use of that kind of Loewner chains of chordal type.
\end{remark}

    %   +++++++++++++++++++++++++++++++++++++
    %
    %
            \subsection{Quasiconformal extensions for Loewner chains}
    %
    %   +++++++++++++++++++++++++++++++++++++

As we mentioned above, in the classical radial case there is an essentially one-to-one correspondence between  Loewner chains $(f_t)$ and classical Herglotz functions $p$. Moreover, it is known that for any conformal map ${f:\UD\Mapinto\Complex}$ with~${f'(0)-1=}{f(0)=0}$, there exists a classical radial Loewner chain~$(f_t)$ such that~$f=f_0$, see~\cite[Theorem\,6.1]{Pom:1975} and~\cite{Gutljanski}.

Therefore, it is a very natural general problem to investigate the interplay between geometric properties of $(f_t)$ and analytic properties of $p$.

 Recall that an orientation-preserving homeomorphism $f$ of a plane domain $G \subset \C$ is said to be \textit{$k$-quasiconformal} (or \textit{k-q.c.}) if $f$ has distributional derivatives ${\de_{z} f := \de f/\de z}$ and ${\de_{\bar{z}} f := \de f/\de \bar{z}}$, both of class $L_{\text{loc}}^1$, satisfying $|\de_{\bar{z}} f| \leq k |\de_{z} f|$ almost everywhere in $G$, where $k \in [0,1)$ is a constant.
We simply say that $f$ is \textit{quasiconformal} if $k$ needs not to be specified. Quasiconformality is invariant under pre- and post-composition with conformal mappings. This allows one to define $k$-q.c. mappings between Riemann surfaces using local coordinates.

For a given conformal mapping $f$ on $G$, $f$ is said to \textit{have a quasiconformal extension} to~$\C$ (or to~$\ComplexE$) if there exists a quasiconformal self-map $\tilde f$ of $\C$ (or of~$\ComplexE$, respectively) whose restriction to $G$ coincides with $f$. Thanks to the removability property for isolated singularities of q.c.-mappings, see, e.g., \cite[Chapter~I, \S8.1]{LehtoVirtanen:1973}, quasiconformal self-maps of $\C$ and $\ComplexE$ are, in fact, automorphisms of~$\C$ and $\ComplexE$, respectively; moreover, $k$-q.c. extendibility to~$\C$ is stronger than that to~$\ComplexE$, because it is equivalent to the existence of a $k$-q.c. extension~$\tilde f$ to~$\ComplexE$ with the additional property that $\tilde f(\infty)=\infty$.

For a comprehensive survey of the theory of quasiconformal mappings in the plane, see \cite{Ahlfors:2006}, \cite{LehtoVirtanen:1973} and \cite[Chapter 4]{ImayoshiTaniguchi:1992}.

In 1972, Becker showed that if for almost every $t \geq 0$, $p(\D, t)$ lies in a compact subset of the right half-plane $\UH :=\{z \in \C : \Re z > 0\}$ not depending on~$t$, then $f_0$ has a q.c.-extension to $\C$. Namely, he proved the following result.

\begin{theorem}[{\cite{Becker:1972, Becker:1976}}]
    %   ---------------------------------------------------

    %       Becker's quasiconformal extension criterion
            \label{TH_Beckerthm}

    %   ---------------------------------------------------
Let $k \in [0,1)$ be a constant.
Suppose that $(f_t)$ is a classical radial Loewner chain for which the classical Herglotz function $p$ in the Loewner\,--\,Kufarev PDE~\eqref{EQ_LK-PDE} satisfies
\begin{eqnarray*}
p(z,t) \in U(k) &:=& \left\{ w \in \C : \left|\frac{w-1}{w+1}\right| \leq k\right\}\\
&=&\left\{ w \in \C : \left|w - \frac{1+k^2}{1-k^2}\right| \leq \frac{2k}{1-k^2}\right\}
\end{eqnarray*}
for all $z \in \D$ and almost all $t\ge 0$.
Then the function $F$ defined by
\begin{equation}
\label{functionF}
F(z) := \left\{
\begin{array}{ll}
f_0(z), & z \in \D,\\
\displaystyle f_{\log |z|}\left(\frac{z}{|z|}\right), & z \in \C\backslash\overline{\D},
\end{array}
\right.
\end{equation}
is a $k$-quasiconformal mapping of $\C$.
\end{theorem}

Many sufficient conditions for univalent functions with quasiconformal extensions have been derived by Theorem \ref{TH_Beckerthm}.
In this paper we will prove an analogue of Becker's Theorem \ref{TH_Beckerthm} for Loewner chains of chordal type and apply it for the proof of several criteria of q.c.-extendibility of holomorphic functions in~$\UH$.

%it is also known that Theorem \ref{TH_Beckerthm} cannot be applied for a Loewner chain $(f_t)$ such that $f_0 \in \mathcal{S}$ is unbounded.
%The reason is that in such a case the above function $F$ is no longer a homeomorphism of $\C$.
%This fact indicates that .

%A large number of sufficient conditions for univalent functions with quasiconformal extensions have been derived by Theorem \ref{TH_Beckerthm}.
%On the other hand, it is also known that Theorem \ref{TH_Beckerthm} cannot be applied for a Loewner chain $(f_t)$ such that $f_0 \in \mathcal{S}$ is unbounded.
%The reason is that in such a case the above function $F$ is no longer a homeomorphism of $\C$.
%This fact indicates that Theorem \ref{TH_Beckerthm} does not suit the chordal situation.

\section{Quasiconformal extendibility of Loewner chains and evolution families}
\label{S_qc-ext-of-Loewner_chains-and-evol_fam}
\subsection{General case}

We start with a simple result that reveals a direct relation between q.c.-extendibility of evolution families and associated Loewner chains in the general case.
We say that a family of functions~$\mathcal F$ is \textit{uniformly q.c.-extendible} if there is~$k\in[0,1)$ such that every $f\in\mathcal F$ has a $k$-q.c. extension to~$\ComplexE$. Note that here we do not require the q.c.-extensions to fix the point~$\infty$.
\begin{proposition}\label{PR_QC-EF-LCh}
Let $(\varphi_{s,t})$ be an evolution family in~$\UD$ or~$\UH$. Then the following three statements are equivalent:
\begin{itemize}
\item[(i)] $(\varphi_{s,t})$ is uniformly q.c.-extendible;
\item[(ii)] there exists a uniformly q.c.-extendible Loewner chain~$(f_t)$ associated with~$(\varphi_{s,t})$;
\item[(iii)] some (and hence any) range-normalized Loewner chain associated with~$(\varphi_{s,t})$ is uniformly q.c.-extendible.
\end{itemize}
\end{proposition}
\begin{proof}
If $(f_t)$ is a Loewner chain such that $f_t$ is $k$-q.c. extendible for each~$t\ge0$, then clearly $\varphi_{s,t}=f_t^{-1}\circ f_s$ is $k'$-q.c. extendible with some $k'\le2k/(1+k^2)$ for each $s\ge0$ and $t\ge s$. Hence (ii) implies (i).

Now suppose that all $\varphi_{s,t}$'s are $k$-q.c. extendible. We use the construction of the standard Loewner chain~$(f_t)$ associated with~$(\varphi_{s,t})$ given in~\cite[Proof of Theorem~3.3]{MR2789373}. For each $s\ge0$, we have $f_s=\lim_{t\to+\infty}L_t\circ\varphi_{s,t}$ for a certain family~$(L_t)$ of M\"obius transformations of~$\ComplexE$. Taking into account that all $k$-q.c. extendible elements of a compact family of univalent holomorphic functions form a compact subfamily, see, e.g., \cite[Theorem 14.1 on p.\,148]{Schober:1975}, we see that $f_s$ is also $k$-q.c. extendible to~$\ComplexE$. This shows that~(i) implies~(iii). The proof is complete.
\end{proof}
\begin{remark}\label{RM_largerconstant}
Note that in the above proposition, the uniform q.c.-extendibility of an evolution family~$(\varphi_{s,t})$ implies that of the associated range-normalized Loewner chains~$(f_t)$ with the {\it same upper bound for the dilatation}. However, in the converse implication the upper bound of the dilatation for~$(\varphi_{s,t})$ must be {\it larger} than that for~$(f_t)$. This can be easily seen from the following example.

\begin{example}
Fix $k\in(0,1)$ and set $f_t(z):=a(t)z/(1-b(t)z^2)$ for all $z\in\UD$ and all~$t\ge0$, where $a(t):=1+t/(1-k)$ and $b(t):=\max\{-k,\,k-t\}$. Note that the functions $g_t(\zeta):=1/f_t(1/\zeta)$, $\zeta\in\Delta:=\ComplexE\setminus\overline\UD$, map the exterior of~$\UD$ conformally onto exteriors of a nested family of ellipses, from which it is easy to see $(f_t)$ is a Loewner chain. Moreover, for each $t\ge0$, $g_t$ extends to a $|b(t)|$-q.c. map of~$\ComplexE$ by setting~$g_t(\zeta):=(\zeta-b(t)\overline\zeta)/a(t)$ for all~$z\in\overline\UD$. Hence all~$f_t$'s are $k$-q.c. extendible, with $f_t(\infty)=\infty$. Now let $\varphi_{s,t}:=f_t^{-1}\circ f_s$, $t\ge s\ge0$. Elementary calculations show that the Schwarzian derivative of~$\varphi_{s,t}$ equals
$$
S\varphi_{s,t}=Sf_s-\big((Sf_t)\circ\varphi_{s,t}\big)(\varphi'_{s,t})^2.
$$
In particular, $S\varphi_{0,2k}(0)=12k(1+k^2)/(1+k)^2$. Using the upper estimate of $Sf(0)$ for q.c.-extendible normalized univalent holomorphic functions in~$\UD$~\cite{Kuehnau:1969, Lehto:1971}, see also \cite[Corollary~14.7 on p.\,152]{Schober:1975}, we conclude that $\varphi_{0,2k}$ can have a $k'$-q.c. extension to~$\ComplexE$ only if $k'\ge 2k(1+k^2)/(1+k)^2>k$.
\end{example}
In the general case, it is natural to expect that the best value of~$k'$ is $2k/(1+k^2)$. It, however, does not seem to be easy to find an example showing that this constant cannot be decreased.
\end{remark}
%\comm{NBBBB: The remark below is not true in general. Moreover, there is no natural and meaningful way to define a "standard Loewner chain in $\UH$" uniquely. It is always up to a M\"obius transformations of a plane or a half-plane. Maybe we should talk about a "geometrically normalized" or "range-normalized" Loewner chains.}
%\begin{remark}\label{RM_remove}
%\dv{Proposition~\ref{PR_QC-EF-LCh} and the previous remark remain  valid if we additionally require that all the q.c.-extensions under consideration fix the point~$\infty$.}
%\end{remark}

\begin{remark}
In~2007 Kuznetsov obtained a sufficient condition \cite[(9) in Theorem~2]{Kuznetsov} for the elements of an evolution family generated by the classical radial Loewner\,--\,Kufarev ODE to map the unit disk onto quasidisks with rectifiable boundary. Unfortunately the method used in~\cite{Kuznetsov} leads only to  a rough estimate for the dilatation ${K:=(1+k)/(1-k)}$, which explodes as $t\to+\infty$. Therefore, it seems worth to mention that, in fact, Kuznetsov's condition implies that the classical Herglotz function $p$ takes values in some compact set $K\subset\UH$ and hence, according to Becker's Theorem~\ref{TH_Beckerthm}, it implies that the corresponding evolution family and the Loewner chain are \textit{uniformly} q.c.-extendible. Indeed, the condition from~\cite{Kuznetsov} can be written as ${\partial \log\Re p(re^{i\theta},t)}/{\partial r}=O((1-r)^{-\alpha})$ as $r\to1^-$ uniformly w.r.t. $t\ge0$ and $\theta\in\Real$ for some~$\alpha\in(0,1)$. Under this condition, the family~$Q_r(e^{i\theta}):=\log\Re p(re^{i\theta},t)$, $\theta\in\Real$, converges uniformly as $r\to1^-$. Therefore, $\log\Re p$ and hence $q:=\Re p$ extend continuously to~$\overline\UD$, with $|q(re^{i\theta})-q(e^{i\theta})|\le C_1(1-r)^{1-\alpha}$ for all~$r\in(0,1)$, all~$\theta\in\Real$, and some constant~$C_1>0$.
In particular, $p(\UD)$ lies in a strip and hence~$p'(z)=O((1-|z|)^{-1})$. Combining the two estimates, we get $|q(e^{i\theta_2})-q(e^{i\theta_1})|\le 2C_1(1-r)^{1-\alpha}+C_2|\theta_2-\theta_1|/(1-r)$ for all $r\in(0,1)$, all~$\theta_1,\theta_2\in\Real$, and some constant~$C_2>0$. Optimizing w.r.t.~$r\in(0,1)$ leads to the conclusion that~$q$ is $\frac{1-\alpha}{2-\alpha}$-H\"older continuous. Using the Herglotz representation formula for~$p$ and \cite[Proposition~3.4]{Pommerenke:BoundBehaConfMaps}, we finally see that the function~$p$ itself extends continuously to~$\overline\UD$.
\end{remark}

    %   ++++++++++++++++++++++++++++++++++++++++++++++++++++++
    %
    %       Section 3
    \subsection{Main result and its proof}
    %
    %   ++++++++++++++++++++++++++++++++++++++++++++++++++++++
\label{SS_maintheorem}

In this section we prove the chordal analogue of Becker's Theorem~\ref{TH_Beckerthm}, i.e. Theorem~\ref{TH_main-thrm-in-short}, which can be formulated in a more detail as follows.
\begin{theorem}
    %   ---------------------------------------------------

    %       Main theorem 1
            \label{TH_alldetails}

    %   ---------------------------------------------------
Let $(f_t)$ be a Loewner chain of chordal type with associated Herglotz function~$p$.
Suppose that there exists a constant $k\in[0,1)$ such that
\begin{equation}\label{EQ_U(k)1}
    p(z,t) \in U(k)\quad\text{for all $z\in\UH$ and a.e.~$t\ge 0$}.
\end{equation}
Then:
\begin{enumerate}
\def\labelenumi{(\roman{enumi})}
\item $f_t$ has a continuous extension to~$i\Real$ for all $t\ge0$;
\item moreover, for all $t\ge 0$, $f_t(\UH)$ can be $k$-q.c. extended to~$\ComplexE$ by setting $f_t(\infty):=\infty$ and
$f_t(-x+iy,t):= f_{t+x}(iy)$ for all $x>0$ and all $y\in\Real$;
\item all elements of the evolution family~$(\varphi_{s,t})$ associated with~$(f_t)$ are also $k$-q.c. extendible to~$\ComplexE$ with $\varphi_{s,t}(\infty)=\infty$ for all~$s\ge0$ and $t\ge s$.
\end{enumerate}
\end{theorem}
An explicit $k$-q.c. extension for~$\varphi_{s,t}$ will be given in the proof of the above theorem.
Note also that (iii) does not follow from~(ii), see Remark~\ref{RM_largerconstant}.

\begin{remark}
In the modern literature \textsl{decreasing} analogues of Loewner chains are often considered, see, e.g., \cite{Lawler:book}. In~\cite{contreraslocalduality} a general definition of a decreasing Loewner chain in~$\UD$ has been given. Passing with the help of the Cayley map to the half-plane~$\UH$ and combining assertion~(iii) of the above theorem with \cite[Theorems~4.1 and~4.2]{contreraslocalduality}, we can easily see that \textit{if $p$ is a Herglotz function in~$\UH$ satisfying~\eqref{EQ_U(k)1}, then the (unique) decreasing Loewner chain~$(g_t)$ corresponding to the vector field $G:=p$, has the following property: for each $t\ge0$ the function $g_t$ has a $k$-q.c. extension to~$\ComplexE$ with a fixed point at~$\infty$.}
\end{remark}

%In order to prove our main theorem, we need the following boundary version of the Schwarz lemma.
%\begin{theorem}[{\cite[Theorem 1.2.5]{Abate:1989}}]
    %   ----------------------------------------------------
    %       Boundary version of the Schwarz Lemma
%           \label{vSchwarz}
    %   ----------------------------------------------------
%Let $f: \D \to \D$ be a holomorphic function, and take $\sigma \in \D$ such that
%$$
%\liminf_{z \to \sigma} \frac{1-|f(z)|}{1-|z|} = \alpha < \infty.
%$$
%Then there exists a unique $\tau \in \de\D$ such that for every $z \in \D$
%$$
%\frac{|\tau - f(z)|^2}{1-|f(z)|^2} \le \alpha \frac{|\sigma -z|^2}{1-|z|^2}
%$$
%Equality holds at one point if and only if $f$ is an automorphism of $\D$.
%\end{theorem}

\begin{proof}[\textbf{Proof of Theorem \ref{TH_alldetails}}]

We start with the proof of~(iii). Fix arbitrary $s\ge0$ and $t>s$.
For all $z\in\UH$ and $a\ge0$ we define
$$%
g_{a}(z):=\left\{%
\begin{array}{ll}
\varphi_{s+a,t}(z),& \text{if~}0 \le a\le t-s,\\[.5ex]
z+t-s-a,& \text{if~}a>t-s.
\end{array}
\right.$$%
According to \cite[Proposition~3.7, Theorem~6.6]{BracciCD:evolutionI}, for any fixed $z\in\UH$, the map $a\mapsto g_{a}(z)$ is absolutely continuous on $[0,+\infty)$ and
\begin{equation}\label{EQ_dg/da}
\frac{\partial g_{a}(z)}{\partial a}=-g'_{a}(z)p_t(z,s+a)\quad\text{for a.e. $a\ge0$},
\end{equation}
where $p_t(\cdot,\xi):=p(\cdot,\xi)$ if $0\le\xi\le t$ and $p_t(\cdot,\xi)\equiv1$ for all~$\xi\ge t$. Moreover, for if $b\ge a\ge0$, then $g_a=g_b\circ\psi_{a,b}$, where $\psi_{a,b}\in\Hol(\UH,\UH)$ is given by the formula
$$%
\psi_{a,b}(z):=\left\{%
\begin{array}{ll}
\varphi_{s+a,s+b}(z),& \text{if~}b\le t-s,\\[.5ex]
\varphi_{s+a,t}(z)+b-(t-s),& \text{if~}a\le t-s<b,\\[.5ex]
z+b-a,& \text{if~}a>t-s.
\end{array}
\right.$$%
Note that $\psi_{a,b}\neq\id_\UH$ whenever $b>a\ge0$ and that the Denjoy\,--\,Wolff point of~$\psi_{a,b}$ is~$\infty$. Since $g_b$ is univalent in~$\UH$, it follows that, if $g_b(z_2)=g_a(z_1)$ for some $b>a$ and $z_1,z_2\in\UH$, then $\Re z_2>\Re z_1$. Using this fact, it is easy to show that for any $\rho>0$,
\begin{equation}%
\label{rhophi}
\varphi^\rho_{s,t}(x+iy):=\left\{%
\begin{array}{ll}
g_0(x+\rho+iy),& \text{for~}x>0,~~y\in\Real,\\[.5ex]
g_{-x}(\rho+iy),& \text{for~}x\le0,~~y\in\Real.
\end{array}\right.
\end{equation}%
is an injective map of $\Complex$ into itself. Note also that $\varphi^\rho_{s,t}(z)=\varphi_{s,t}(z+\rho)$ for all~$z\in\UH$. Condition~\eqref{EQ_U(k)1} and equality~\eqref{EQ_dg/da} imply that for all $z\in\Complex$ with $\Re z<0$,
$$
\left|\frac{\partial_{\overline{z}} \varphi_{s,t}^\rho(z)}{\partial_{z} \varphi_{s,t}^\rho(z)}\right|
=
\left|\frac
    {\partial_{x} \varphi_{s,t}^\rho(x + iy) + i \partial_{y} \varphi_{s,t}^\rho(x + iy)}
    {\partial_{x} \varphi_{s,t}^\rho(x + iy) - i \partial_{y} \varphi_{s,t}^\rho(x + iy)}\right|
=
\left|\frac{p_t(\rho+iy,s-x)-1}{p_t(\rho+iy,s-x)+1}\right|\le k.
$$
Taking into account the local absolute continuity of $a\mapsto g_a(z)$ for every fixed $z\in\UH$,  we may now conclude that for any $\rho>0$ the function $\varphi^\rho_{s,t}$ is a $k$-q.c. extension of~$\UH\ni z\mapsto\varphi_{s,t}(z+\rho)$ to~$\Complex$.

Now using the compactness property of normalized $k$-q.c. self-map of~$\Complex$ (see, e.g., \cite[p.74]{LehtoVirtanen:1973}) we can conclude that $\varphi_{s,t}$ has a continuous extension to $i\Real$ and that $\varphi_{s,t}^\rho$ tends, as $\rho\to0^+$, to the $k$-q.c. extension of~$\varphi_{s,t}$ defined for all~$z\in\Complex\setminus\overline{\UH}$ by $\varphi_{s,t}(z):=g_{-\Re z}(i\,\Im z)$. It remains to mention that any q.c. self-map of~$\Complex$ extends to a q.c. automorphism of~$\ComplexE$ simply by setting~$\varphi_{s,t}(\infty)=\infty$.

The proof of (i) and~(ii) follows the same way in reasoning. We fix~$\rho>0$ and define
\begin{equation}%
\label{rhoft}
f^\rho_{t}(x+iy):=\left\{%
\begin{array}{ll}
f_t(x+\rho+iy),& \text{for~}x>0,~~y\in\Real,\\[.5ex]
f_{t-x}(\rho+iy),& \text{for~}x\le0,~~y\in\Real.
\end{array}\right.
\end{equation}%
Since $(t,z)\mapsto f_t(z)$ solves the Loewner PDE $\mydot{f_t}(z)=-f_t'(z)p(z,t)$,  for all~$z\in\UH$ and
a.e.~$t\ge0$, from~\eqref{EQ_U(k)1} we deduce that $f^\rho_t$ is a $k$-q.c. extension of~$z\mapsto f_t(z+\rho)$
to~$\Complex$.

Exactly the same argument as we have used for the family~$\varphi_{s,t}^\rho$ shows that $f_t$
extends continuously to~$i\Real$ and that it further can be extended to a $k$-q.c. automorphism
of~$\ComplexE$ by setting $f_t(\infty)=\infty$ and $f_t(z)=f_{t-\Re z}(i\,\Im z)$ for all
$z\in\Complex\setminus\overline\UH$.
\end{proof}

%\comm{Do we really need the following remark? (I would suggest to delete it.)}
%\begin{remark}
%\dv{In Theorem \ref{TH_alldetails} we may assume that the Herglotz function $p$ on not $\UH$ but the unit disk $\UD$ lies in $U(k)$ for all $z \in \UD$ and almost all $t \ge 0$.
%In this case, since $p_{\UH}(z,t) = 2p (\frac{z-1}{z+1},t)$ for all $z \in \UH$, $x$ of $g_{-x}$ on the right-hand side of \eqref{rhophi} and $f_{t-x}$ on the right-hand side of \eqref{rhoft} for $x \le 0$ should be replaced to $\frac{x}2$.}
%\end{remark}

%\noindent\dv{{\it Remark 3.9}} \comm{I deleted the word "Remark" here.}
A direct corollary of Theorem~\ref{TH_alldetails} is that under condition~\eqref{EQ_U(k)1} the
\textsl{Loewner range} $\cup_{t\ge0}f_t(\UH)$ coincides with the whole plane~$\Complex$. Below
we show that the same conclusion can be deduced under a weaker assumption upon the Herglotz
function~$p$.

\begin{proposition}
    %   ----------------------------------------------------

    %       Proposition: Loewner Range for f_t(H)
            \label{PR_fill_allC_chordal}

    %   ----------------------------------------------------
Let $(f_t)$ be a Loewner chain of chordal type in~$\UH$. Suppose that the Herglotz  function~$p$ associated
with~$(f_t)$ satisfies
\begin{equation}\label{EQ_weaker-cond}
C_1<\Re p(z,t)<C_2\quad\text{ for all $z\in\UH$ and a.e.~$t\ge0$,}
\end{equation}
where $C_1,C_2\in(0,+\infty)$ are constants. Then the Loewner range of $(f_t)$, i.e. the set
$\cup_{t\ge0}f_t(\UH)$, coincides with~$\Complex$.
\end{proposition}
\begin{proof}[\textbf{Proof}]
As it is easy to see (e.g. using the explicit formulas), there exists $\kappa\in(0,1)$  such that
$\lambda_\UH(z)<\kappa\lambda_\Pi(z)$ for all~$z\in\Pi:=\{z\colon C_1<\Re z<C_2\}$, where
$\lambda_D$ stands for the density of the hyperbolic metric in a domain~$D$. Since $p$ is non-expanding as a map
from~$\UH$ to $\Pi$ endowed with the corresponding hyperbolic metrics, it follows that
\begin{equation}\label{EQ_p_UH}
\frac{\Re p'(z)}{\Re p(z)}\le
\frac{|p'(z)|}{\Re p(z)}\le\frac{\kappa}{\Re z}\quad \text{for all $z\in\UH$}.
\end{equation}

Let $(\varphi_{s,t})$ stand for the evolution family associated with~$(f_t)$. Then
it satisfies the Loewner\,--\,Kufarev ODE of the form $(d/dt) \varphi_{s,t}(z)=p(\varphi_{s,t}(z),t)$, from which we get
\begin{equation}\label{EQ_alpha}
\alpha(t):=\log\frac{|\varphi_{0,t}'(1)|}{\Re \varphi_{0,t}(1)}=2\int_0^t\left(\Re p'\big(\varphi_{0,s}(1),s\big)- \frac{\Re p\big(\varphi_{0,s}(1),s\big)}{\Re \varphi_{0,s}(1)}\right)\,\mathrm ds
\end{equation}for all $t\ge0$.
Again using the Loewner ODE, from inequality~\eqref{EQ_weaker-cond} it follows that
\begin{equation}\label{EQ_growth-of-Phi}
\frac{\Re p\big(\varphi_{0,s}(1),s\big)}{\Re \varphi_{0,s}(1)}\ge \frac{C_1}{2tC_2}\quad\text{for a.e. $s\ge0$}.
\end{equation}
Combining \eqref{EQ_alpha} with the estimates \eqref{EQ_p_UH} and~\eqref{EQ_growth-of-Phi}, we conclude that $\alpha(t)\to-\infty$ as $t\to+\infty$. Hence, by \cite[Theorem~1.6]{MR2789373}, $\cup_{t\ge0}f_t(\UH)=\Complex$.
\end{proof}

\begin{remark}\label{RM_essential-both-ineq}
Two simple examples of chordal Loewner chains in~$\UH$  with the Loewner range different from~$\Complex$,
namely $f^{1}_t(z):=-t+\log(z+1)$ with associated Herglotz function $p_1(z,t):=z+1$ and $f^2_t(z):=z-\arctan t$
with associated Herglotz function $p_2(z,t):=(1+t^2)^{-1}$, show that both inequalities of condition~\eqref{EQ_weaker-cond} are essential in Proposition~\ref{PR_fill_allC_chordal}.
\end{remark}

\begin{remark}\label{RM_applyAut}
 If $p$ is a Herglotz function in~$\UH$, then $\lambda(t):=\int_{0}^tp'(\infty,\xi)\,\mathrm d\xi$ is well-defined for all~$t\ge0$, see \cite[proof of Theorem~7.1, p.\,29--30]{BracciCD:evolutionI}, and moreover, according to the Julia\,--\,Wolff\,--\,Carath\'eorody Theorem  for the half-plane, see, e.g.~\cite[Ch.\,IV~\S26]{Valiron},  $\tilde p(z,t):=e^{-\lambda(t)}p(e^{\lambda(t)} z,t)-p'(\infty,t)z$, defined for all $z\in\UH$ and~$t\ge0$, is also a Herglotz function in~$\UH$. Furthermore, it is not difficult to see that if~$(\varphi_{s,t})$ is the evolution family associated to~$p$, then the functions $\tilde\varphi_{s,t}(z):=e^{-\lambda(t)}\varphi_{s,t}(e^{\lambda(s)} z)$ form the evolution family associated with~$\tilde p$. Finally, if $(\tilde f_t)$ is a Loewner chain associated with~$(\tilde\varphi_{s,t})$, then using~\cite[Lemma~3.2]{MR2789373} we see that the functions~$f_t(z):=\tilde f_t(e^{-\lambda(t)}z)$ form a Loewner chain associated with~$(\varphi_{s,t})$. Since trivially $f_t(\UH)=\tilde f_t(\UH)$ for all~$t\ge0$, it follows that the conclusion of Proposition~\ref{PR_fill_allC_chordal} holds also under a weaker condition that~\eqref{EQ_weaker-cond} takes place with $\tilde p$ substituted for~$p$.
\end{remark}

\subsection{Case of unbounded Herglotz function}\label{SS_case-of-unbdd}
Using the trick of Remark~\ref{RM_applyAut}, from Theorem~\ref{TH_alldetails} we deduce the following assertion.

\begin{corollary}\label{C_qc-ext}
Let $p$ be a Herglotz function in~$\UH$ and let $k\in[0,1)$. Suppose that there exist measurable functions $\alpha:[0,+\infty)\to[0,+\infty)$ and $\beta:[0,+\infty)\to\Real$ such that
\begin{equation}\label{EQ_U(k)-unbdd}
\alpha(t)\big[p(z,t)-p'(\infty,t)z-i\beta(t)\big]\in U(k)\quad\text{for all~$z\in\UH$ and a.e.~$t\ge0$}.
\end{equation}
Then:
\begin{itemize}
\item[(i)] elements of the chordal evolution family~$(\varphi_{s,t})$ associated with~$p$ are $k$-q.c. extendible to~$\ComplexE$;
\item[(ii)] elements of any range-normalized chordal Loewner chain~$(f_t)$ associated with~$p$ are $k$-q.c. extendible to~$\ComplexE$.
\end{itemize}
\end{corollary}
\begin{proof}
First of all note that we may assume $p'(\infty,t)=0$ for all~$t\ge0$. Indeed, otherwise we should consider $\tilde p$, $(\tilde\varphi_{s,t})$, $(\tilde f_t)$, defined in Remark~\ref{RM_applyAut}, and $e^{\lambda(t)}\alpha$, where $\lambda(t):=\int_{0}^tp'(\infty,\xi)\,\mathrm d\xi$, instead of $p$, $(\varphi_{s,t})$, $(f_t)$ and $\alpha$, respectively.

Recall that by definition of a Herglotz function, $t\mapsto p(1,t)$ is locally integrable on~$[0,+\infty)$. Hence from~\eqref{EQ_U(k)-unbdd} it follows that $1/\alpha$ and $\beta$ are also locally integrable functions on~$[0,+\infty)$. Let $u(t):=\int_0^t(1/\alpha(s))\mathrm ds$ and $v(t):=\int_0^t \beta(s)\mathrm ds$ for all~$t\ge0$. Now we define a Herglotz function $\hat p:\UH\times[0,+\infty)\to\Complex$ as follows. For all $\xi\in[0,T)$, where $T:=\sup_{t\ge0} u(t)$, and all $z\in\UH$ we set $$\hat p(z,\xi):=\alpha(t)\big[p\big(z+iv(t),\,t\big)-i\beta(t)\big]\Big|_{t:=u^{-1}(\xi)}.$$ If $T<+\infty$, we set $\hat p(\cdot,\xi)\equiv1$ for all~$\xi\ge T$. Note that $\hat p$ is measurable in~$t$ and satisfies~\eqref{EQ_U(k)1}. Therefore, $\hat p$ is indeed a Herglotz function in~$\UH$ and, by Theorem~\ref{TH_alldetails}, all the elements of the evolution family~$(\hat\varphi_{s,t})$ associated with~$\hat p$ are $k$-q.c. extendible.

To prove~(i), it remains to notice that $\varphi_{s,t}(z)=\hat\varphi_{u(s),u(t)}(z-iv(s))+iv(t)$ whenever $z\in\UH$ and $t\ge s\ge0$ by uniqueness of solutions to the Loewner~ODE. Now (ii) follows from~(i) by Proposition~\ref{PR_QC-EF-LCh}.
\end{proof}

%\comm{I removed the reference to Remark~\ref{RM_remove} at the end of the proof. By $k$-q.c. extendible we will mean having a $k$-q.c. extension to $\ComplexE$ without requirement to fix~$\infty$.}

\begin{remark}
The q.c.-extension of $(\varphi_{s,t})$ in the corollary we proved above has a fixed point at~$\infty$, because this is the DW-point of~$(\varphi_{s,t})$. At the same time, in contrast to Theorem~\ref{TH_alldetails}, it is not guaranteed the q.c.-extension of~$(f_t)$ fixes~$\infty$.
\end{remark}

    %   ++++++++++++++++++++++++++++++++++++++++++++++++++++++
    %
    %       Section ?
     \section{Loewner chain criteria}\label{S_LCh-criteria}
      %
    %
    %   ++++++++++++++++++++++++++++++++++++++++++++++++++++++

Several classical sufficient conditions for univalence in~$\UD$ can be obtained as corollaries of Pommerenke's criterion~\cite[Theorem~6.2]{Pom:1975} for a solution of the (radial classical) Loewner\,--\,Kufarev PDE to be a Loewner chain, see, e.g., \cite[\S6.3]{Pom:1975}.
The theorem below can be regarded as an analogue of that criterion for general Loewner chains. The statement of the theorem uses the following definition.
\begin{definition}
Let $D\subset\ComplexE$ be a domain with~$\mathop{\mathsf{Card}}(\ComplexE\setminus D)\ge3$. A family~$\mathcal F$ of functions from~$D$ to~$\ComplexE$ is said to be \textsl{uniformly locally univalent} in~$\UD$ if there exists~$\rho>0$ such that every~$f\in\mathcal F$ is univalent in every hyperbolic disk\footnote{That is a ball w.r.t. the hyperbolic distance in~$D$.} of radius~$\rho$ in~$D$.
\end{definition}

\begin{theorem}\label{TH_suff-for-Loewner-chain}
Let $(g_t)_{t\ge0}$ be a family of holomorphic functions in~$\UD$. Suppose that the following conditions hold:
\begin{itemize}
\item[(i)] the map $\UD\times[0,+\infty)\ni (z,t)\mapsto g_t(z)$ is continuous and for each~$z\in\UD$ the map $[0,+\infty)\ni t\mapsto g_t(z)$ is locally absolutely continuous;
\item[(ii)] there exists a Herglotz vector field~$G:\UD\times[0,+\infty)$  such that for every~$z\in\UD$,
$$\frac{\partial g_t(z)}{\partial t}=-g'_t(z)G(z,t)\quad\text{for a.e.~$t\ge0$};$$

\item[(iii)] the Loewner range $\cup_{t\ge0}f_t(\UD)$ of some (an hence any) chordal Loewner chain associated with~$p$ coincides with~$\Complex$;
\item[(iv)] the family~$(g_t)$ is uniformly locally univalent in~$\UD$.
\end{itemize}
Then $(g_t)$ is a Loewner chain associated with the Herglotz vector field~$G$.
\end{theorem}
Note that the above theorem is conformally invariant and can be also stated in the framework of the right half-plane. We, however, will use its more specific version.

\begin{theorem}\label{Th_suff-for-chordal-Loewner-chain}
Let $x_0>0$ and $(h_t)_{t\ge0}$ be a family of holomorphic functions in~$\UH(x_0):=\{z\colon\Re z>x_0\}$ satisfying the following conditions:
\begin{itemize}
\item[(i)] the map $\UH(x_0)\times[0,+\infty)\ni (z,t)\mapsto h_t(z)$ is continuous and for each~$z\in\UH(x_0)$ the map $[0,+\infty)\ni t\mapsto h_t(z)$ is locally absolutely continuous;
\item[(ii)] there exist constants $0<C_1<C_2<+\infty$ and a Herglotz function\\ $p:\UH\times[0,+\infty)\to\Pi:=\{\zeta\colon C_1<\Re \zeta<C_2\}$ such that for each~$z\in\UH(x_0)$,
\begin{equation}\label{EQ_LoewPDE_x_0}
\frac{\partial h_t(z)}{\partial t}=-h_t'(z)p(z,t)\quad \text{for a.e. $t\ge0$}.
\end{equation}

\item[(iii)] there exists $a>x_0$  such that $$\liminf_{t\to+\infty}\inf_{\Re z>a+C_1t}\frac{\mathsf{R}_u(h_t,z)}{\Re z}>0,$$ where $\mathsf{R}_u(f,z)$ stands for the (Euclidean) radius of univalence of a function~$f$ at the point~$z$.
\end{itemize}
Then for every~$t\ge0$ the function $h_t$ extends to a univalent holomorphic function ${g_t:\UH\to\Complex}$, with $(g_t)_{t\ge0}$ being a Loewner chain associated with the Herglotz function~$p$.
\end{theorem}

The proofs of the two theorems stated above are based on very similar arguments. That is why we present here only the proof of the latter one.

\begin{proof}[\textbf{Proof of Theorem~\ref{Th_suff-for-chordal-Loewner-chain}}]
Let $(\varphi_{s,t})$ be the evolution family associated with the Herglotz function~$p$. Note that $\varphi_{s,t}(\UH(x_0))\subset\UH(x_0)$ whenever $t\ge s\ge0$ thanks to~\eqref{EQ_JWC}.
Since $(z,t)\mapsto h_t(z)$ solves the PDE~\eqref{EQ_LoewPDE_x_0} in the sense of \cite[Definition 2.1]{contreraslocalduality}, it follows, see, e.g., \cite[Proposition 2.3]{contreraslocalduality}, that
\begin{equation}\label{EQ_varphi_h}
h_t\circ\varphi_{s,t}|_{\UH(x_0)}=h_s\quad\text{for all $s\ge0$ and all $t\ge s$.}
\end{equation}
Let $(f_t)$ be a Loewner chain associated with~$(\varphi_{s,t})$. Then $\varphi_{s,t}\circ f_s^{-1}|_{\Omega(s)}=f_t^{-1}|_{\Omega(s)}$ whenever $t\ge s\ge0$, where $\Omega(s):=f_s(\UH(x_0))$. Combining the latter equality with~\eqref{EQ_varphi_h}, we see that $h_s\circ f_s^{-1}|_{\Omega(s)}=h_t\circ f_t^{-1}|_{\Omega(s)}$ for all $s\ge0$ and $t\ge s$. Therefore, there exists a holomorphic function $\Psi:\Omega\to\Complex$, $\Omega:=\cup_{s\ge0}\Omega(s)$, such that \begin{equation}\label{EQ_ht-ft}
h_t=\Psi\circ f_t|_{\UH(x_0)}\quad\text{for all $t\ge0$}.
\end{equation}

It is easy to see that the functions $\tilde f_t(z):=f_t(z+x_0)$, $z\in\UH$, $t\ge0$, form a Loewner chain associated with the evolution family~$(\tilde\varphi_{s,t})$ given by $\tilde\varphi_{s,t}(z):=\varphi_{s,t}(z+x_0)-x_0$, and with the Herglotz function $\tilde p(z,t):=p(z+x_0,t)$. Therefore, by  Proposition~\ref{PR_fill_allC_chordal}, \begin{equation}\label{EQ_Omega=all-C}
\Omega=\bigcup_{t\ge0}\tilde f_t(\UH)=\Complex.
\end{equation}

Fix $s\ge0$ and denote $b:=a+C_1 s+1$, $z(t):=\varphi_{s,t}(b)$. On the one hand, by~(ii), $\Re z(t)>a+C_1t$ for all~$t\ge s$ and hence by (iii) there exist $T_0\ge s$ and $\rho>0$ such that $h_t$ is univalent in~$B_t:=\{z:|z-z(t)|<\rho\, \Re z(t)\}$ for all $t\ge T_0$. On the other hand, the proof of Proposition~\ref{PR_fill_allC_chordal} shows that $\varphi_{s,t}'(b)/\Re z(t)\to0$ as $t\to+\infty$. Since the function~$\varphi_{s,t}$ is univalent in~$\UH$ for each~$t\ge s$, it follows that for any compact set~$K\subset \UH(x_0)$ there exists $T(K)>T_0$ such that $\varphi_{s,t}(K)\subset B_t$ for all $t>T(K)$. As a result, the composition $h_t\circ \varphi_{s,t}$ is injective on~$K$ for all $t>T(K)$, which in view of~\eqref{EQ_varphi_h} implies that~$h_s$ is injective on every compact set~$K\subset \UH(x_0)$ and hence univalent in $\UH(x_0)$. Since $s\ge0$ is arbitrary in this argument, from~\eqref{EQ_ht-ft} and~\eqref{EQ_Omega=all-C} it follows that the functions $g_t:=\Psi\circ f_t$ are well-defined univalent holomorphic extensions of~$h_t$ to~$\UH$, for all~$t\ge0$, and that they form a Loewner chain in~$\UH$ associated with the Herglotz function~$p$. This proves the theorem.
\end{proof}

    %   ++++++++++++++++++++++++++++++++++++++++++++++++++++++
    %
    %       Section ?
      \section{Sufficient conditions for quasiconformal extension}
      \label{S_sufficient-conditions}
      %
    %
    %   ++++++++++++++++++++++++++++++++++++++++++++++++++++++
    %
%
We demonstrate usage of Theorems~\ref{TH_alldetails} and~\ref{Th_suff-for-chordal-Loewner-chain} by proving two known sufficient conditions for univalence and quasiconformal extendibility involving $h''/h'$ and $Sh$.
\begin{theorem}[\cite{BecPom:1984}, see also~{\cite[p.\,183]{AksentevS:2002}}]\label{TH_sufficient-uni-qc}
If a non-constant holomorphic function $h:\UH\to\Complex$ satisfies the condition
\begin{equation}\label{EQ_cond-second-der}
\left|\frac{h''(z)}{h'(z)}\right|\le \frac{k}{2\,\Re z}\quad \text{for all $z\in\UH$}
\end{equation}
and some $k\in[0,1]$, then $h$ is univalent in~$\UH$. Moreover, if $k<1$, then $h$ has a $k$-q.c. conformal extension to~$\ComplexE$ with $h(\infty)=\infty$.
\end{theorem}
\begin{proof}
Note that by~\eqref{EQ_cond-second-der}, $h'$ does not vanish in~$\UH$.
We assume first that $k<1$. Following~\cite{BecPom:1984}, consider the family $h_t(z):=h(z+t)-2th'(z+t)$ for all $t\ge0$ and $z\in\UH$. It is easy to see that $(h_t)$ satisfies condition (i) and equation~\eqref{EQ_LoewPDE_x_0} in Theorem~\ref{Th_suff-for-chordal-Loewner-chain} with any $x_0>0$ and with $$p(z,t):=\frac{h'(z+t)+2th''(z+t)}{h'(z+t)-2th''(z+t)},\qquad t\ge0,~~z\in\UH.$$
From~\eqref{EQ_cond-second-der} it follows  that $p$ is a Herglotz function in~$\UH$ satisfying condition~\eqref{EQ_U(k)1} in Theorem~\ref{TH_alldetails}.  In particular, $C_1:=1/(2K)<\Re p<C_2:=2K$ , $K:=(1+k)/(1-k)$. This implies condition~(ii) in Theorem~\ref{Th_suff-for-chordal-Loewner-chain}.

Consider now the functions $f_{z,t}(\zeta):=h_t(z+\zeta\,\Re z/2)-h_t(z)$, $t\ge0$, $z\in\UH$, $\zeta\in\UD$. Applying~\eqref{EQ_cond-second-der} thrice, we obtain
$$
\left|\frac{f'_{z,t}(\zeta)}{f'_{z,t}(0)}\right|<K\left|\frac{h'(z+t+\zeta\,\Re z/2)}{h'(z+t)}\right|<Ke^{k/2}
$$
for any $z\in\UH$, any $t\ge0$ and any~$\zeta\in\UD$. It follows  that
 the family ${\{f_{z,t}/f_{z,t}'(0)\colon z\in\UH\}}$ is normal in~$\UD$. Therefore, (iii)~in Theorem~\ref{Th_suff-for-chordal-Loewner-chain} is also fulfilled, with any~$a>x_0$. Thus $(h_t)$ is a Loewner chain of chordal type satisfying the hypothesis of Theorem~~\ref{TH_alldetails}. This proves Theorem~\ref{EQ_cond-second-der} for~$k<1$.

To give the proof for $k=1$ we consider the family of functions in~$\UH$ defined by $$h(z,\varepsilon):=h(1)+\int_{1}^{z}h'(1)\big(h'(\zeta)/h'(1)\big)^{1-\varepsilon}\,\mathrm{d}\zeta,\quad z\in\UH,~~\varepsilon\in(0,1),$$
where we choose the branch of $(h'(\zeta)/h'(1))^{1-\varepsilon}$ that equals~$1$ at~$\zeta=1$.
By the above argument, $h(\cdot,\varepsilon)$ is univalent in~$\UH$ for any $\varepsilon\in(0,1)$. Since $h(\cdot,\varepsilon)\to h$ as $\varepsilon\to 0^+$ and since $h$ is non-constant by the hypothesis, we may conclude that $h$ is also univalent in~$\UH$. This finishes the proof.
\end{proof}

\begin{example}
It is easy to see that Theorem~\ref{TH_sufficient-uni-qc} can be applied to any Schwarz\,--\,Christoffel integral $h$ that maps~$\UH$ onto a convex unbounded polygon having angle~$\theta\in(0,\pi)$ at infinity or onto the exterior of such a polygon and satisfies the normalization $h(\infty)=\infty$. Theorem~\ref{TH_sufficient-uni-qc} guarantees $k$-q.c. extendibility of such a map~$h$ with~$k:=1-\theta/\pi$.
\end{example}

It seems interesting to compare the sufficient condition given by Theorem~\ref{TH_sufficient-uni-qc} with the following necessary condition.
\begin{proposition}
If $h:\UH\to\Complex$ is univalent, then
\begin{equation}\label{EQ_PreScw-in-H}
\left|\frac{h''(z)}{h'(z)}\right|\le\frac{3}{\Re z}\quad\text{for all $z\in\UH$}.
\end{equation}
Moreover, if $h$ admits a $k$-q.c. extension to~$\ComplexE$ with the fixed point at~$\infty$, then
\begin{equation}\label{EQ_PreScw-in-H-qc}
\left|\frac{h''(z)}{h'(z)}\right|\le\frac{3k}{\Re z}\quad\text{for all $z\in\UH$}.
\end{equation}
\end{proposition}
\begin{proof}
Applying the inequality $(1-|\zeta|^2)|f''(\zeta)/f'(\zeta)|\leq 6$, see, e.g.,~\cite[p.\,41--42]{Becker:1980}, to $f(\zeta):=\big(f_{z,r}(\zeta)-f_{z,r}(0)\big)/f_{z,r}'(0)$, where $f_{z,r}(\zeta):=h((\zeta+1)\tfrac{\Re z}{1-r}+i\,\Im z)$, ${\zeta\in\UD}$, ${z\in\UH}$, ${r\in[0,1)}$, at $\zeta:=-r$ and letting $r\to1^-$ we obtain~\eqref{EQ_PreScw-in-H}.

If now $h$ has a $k$-q.c. extension to~$\ComplexE$ with the fixed point at~$\infty$, then $f$ belongs to the class $\clS_{k,R}$ considered in~\cite[p.\,354]{Lehto:1976} for all $R>1$. Therefore, letting~$R\to+\infty$, from Lehto's Majorant Principle~\cite{Lehto:1976} we deduce that~$6$ in the estimate for~$f''/f'$ can be replaced by~$6k$. This proves~\eqref{EQ_PreScw-in-H-qc}.
\end{proof}
\begin{remark}
The example $h(z):=1/z^{1+k}$, $k\in[0,1]$, shows that~\eqref{EQ_PreScw-in-H} is sharp  and that the requirement for the q.c.-extension to fix~$\infty$ is essential: for $k\in[0,1)$, $h$ is $k$-q.c. extendible by Theorem~\ref{TH_sufficient-uni-qc} applied to~$1/h$,  but $|h''(1)/h'(1)|=2+k>3k$.
\end{remark}

It is well-known that the following condition, a q.c.-analogue of Nehari's univalence criterion,
\begin{equation}\label{EQ_KN-disk}
|Sf(\zeta)|\,(1-|\zeta|^2)^2\le 2 k\quad\text{for all~$\zeta\in\UD$},
\end{equation}
is sufficient for $k$-q.c. extendibility of a non-constant holomorphic function $f:\UD\to\Complex$, see, e.g.,  \cite[ineq.\,(2)]{Ahlfors:1974}. Thanks to properties of the Schwarzian, this statement is equivalent to the fact that the condition
\begin{equation}\label{EQ_KN-half-plane}
|Sh(z)|\,\,(\Re z)^2\le k/2\quad\text{for all~$z\in\UH$},
\end{equation}
is sufficient for $k$-q.c. extendibility of a non-constant holomorphic function $h:\UH\to\Complex$.

For the case of~$\UD$, Becker's Theorem~\ref{TH_Beckerthm} yields an explicit formula for a q.c.-extension of functions~$f$ satisfying~\eqref{EQ_KN-disk}, see, e.g.,~\cite{Becker:1972,Becker:1980}.

Similarly, Theorems~\ref{TH_alldetails} and~\ref{Th_suff-for-chordal-Loewner-chain} allow us to obtain analogous formula for~$\UH$. Note that since the Schwarzian is invariant w.r.t. post-composition with M\"obius transformations, we may assume that $h(z)\to\infty$ as ${\UH\ni z\to\infty}$.

\begin{proposition}\label{PR_Schwarzian}
Let $h:\UH\to\Complex$ be a non-constant holomorphic function satisfying~\eqref{EQ_KN-half-plane} for some~$k\in[0,1)$. Then $h$ has a $k$-q.c. extension to~$\ComplexE$.

Moreover, if $h(z)\to\infty$ as $\UH\ni z\to\infty$, then the  formula
\begin{equation}
\label{EQ_ext}
\tilde h(z):=\left\{
\begin{array}{ll}
h(z)\vphantom{\displaystyle\int_0^1}, & \text{if~$\Re z\ge0$},\\
h(z^*)+\dfrac{2 h'(z^*)\,\Re z}{1-\frac{h''(z^*)}{h'(z^*)}\,\Re z},~z^*:=-\overline z, & \text{if~$\Re z<0$},
\end{array}\right.
\end{equation}
defines a $k$-q.c. extensions of~$h$.
\end{proposition}
\begin{proof}
Note that $Sh=(Ph)'-(Ph)^2/2$, where $Ph:=h''/h'$. Consider the integral operator
$$
Lg(x+iy):=-\int\limits_{x}^{+\infty}\!\Big(\tfrac{1}{2}g(x'+iy)^2+Sh(x'+iy)\Big)\mathrm{d}x',\quad x+iy\in\UH,
$$
defined on the set~$W$ of all holomorphic $g:\UH\to\Complex$ with finite ``pre-Schwarzian'' norm $\|g\|_P:=\sup_{z\in\UH}|g(z)|\,\Re z$. With the help of~\eqref{EQ_KN-half-plane} it is easy to see that $L$ is a contracting self-map of~$\mathbb B_k:=\{g:\|g\|_P\le k\}$. Therefore, there exists $g_0\in \mathbb B_k$ such that $Sh=g_0'-g_0^2/2$. Let $\hat h$ be the unique holomorphic function in~$\UH$ satisfying $\hat h(1)=h(1)$, $\hat h'(1)=h'(1)$, and $P\hat h=g_0$. Since by construction, $S\hat h=Sh$, we have $h=T\circ \hat h$ for some M\"obius transformation~$T$. Therefore, to prove the first part of the proposition, it is sufficient to show that $\hat h$ has a $k$-q.c. extension.
 To this end we will demonstrate that \eqref{EQ_ext} would define a $k$-q.c. extension of~$\hat h$ if we replace $h$ with~$\hat h$ in that formula.  Consider the family
$$
h_t(z):=\hat h(z+t)-\frac{2t\hat h'(z+t)}{1+tP\hat h(z+t)},\quad z\in\UH,~~t\ge0.
$$
Note that $$h_t'(z)=\hat h'(z+t)\frac{1+2t^2 S\hat h(z+t)}{(1+t P\hat h(z,t))^2}\quad\text{for all~$z\in\UH$ and all~$t\ge0$}.$$
Taking  into account the inequality~$\|P\hat h\|_P\le k$ and making use of essentially the same method as in the proof of Theorem~\ref{TH_sufficient-uni-qc}, it is not difficult to see that $(h_t)$ satisfies the hypothesis of Proposition~\ref{PR_LoewChain} with $p(z,t):=\big(1-2t^2 S\hat h(z+t)\big)/\big(1+2t^2 S\hat h(z+t)\big)$. Thus by Theorem~\ref{TH_main-thrm-in-short}, $\hat h$ has a $k$-q.c. extension with the fixed point at~$\infty$ given by formula~\eqref{EQ_ext} in which we have to replace~$h$ by~$\hat h$.

To complete the proof it remains to observe that if  $h(z)\to\infty$ as $z\to\infty$, then $\hat h=h$.
\end{proof}

\subsection{Quasiconformal extendibility via one-parameter semigroups}

In this subsection we prove two sufficient conditions for q.c.-extendibility of holomorphic functions in~$\UH$, using Theorem~\ref{TH_alldetails} and the theory of one-parameter semigroup of holomorphic functions.

\begin{theorem}\label{TH_sufficient-qc1}
Let $h\colon\UH\to\Complex$ be a holomorphic function. Suppose that $h'(\UH)$ is contained in a closed Euclidean disk~$B$, with $0\not\in B$.
Then $h$ admits a $k$-q.c. extension to~$\ComplexE$ with a fixed point at~$\infty$, where $k:=(K-1)/(K+1)$, $K:=\max\limits_{w,z\in B}\sqrt{|w/z|}$.
\end{theorem}

Prior to the proof of the above theorem, we would like to make some comments.

\begin{remark}\label{RM_explicit}
Sugawa~\cite[Theorem~4.1]{SugawaPol} proved a more general result, which states essentially that Theorem~\ref{TH_sufficient-qc1} holds with $\UH$ replaced by any convex domain. The proof we give below is based on Theorem~\ref{TH_alldetails} and therefore works only for~$\UH$. However, it has an advantage that it immediately provides the explicit formula for a $k$-q.c. extension~$\tilde h$ of~$h$ to~$\Complex$. Namely, if $w_1$ and $w_2$ are the two points of~$B$ at which $B\ni w\mapsto|w|$ attains its extrema, then the desired extension is given by $\tilde h(x+iy)=h(iy)+\omega x$ for all $y\in\Real$ and all $x<0$, where $\omega:=\sqrt{w_1w_2}\in B$. One obvious property of this extension is used in the proof of Corollary~\ref{C_qc-sufficient}.
\end{remark}

\begin{example}\label{EX_h}
Let $h(z):=z-a/(1+az)$ for all~$z\in\UH$, where $a>0$ is a constant. Elementary but a bit laborious computations show that $h'(\UH)$ is bounded, with $|h'(z)|>b$ for all~$z\in\UH$ and some constant $b>0$ depending on~$a$. Therefore, by Theorem~\ref{TH_sufficient-qc1}, $h$ admits a q.c.-extension. At the same time $|h''(1)/h'(1)|=a^3/(4+a^2)$ implies that $h$ does not satisfy the Becker\,--\,Pommerenke condition~\eqref{EQ_cond-second-der} for all~$a\ge 2$.  Similarly, the q.c.-analogue of Nehari's condition, i.e. condition~\eqref{EQ_KN-half-plane}, fails at~$z:=1$ provided $a$ is large enough.
\end{example}

\begin{proof}[\fontseries{bx}\selectfont Proof of Theorem~\ref{TH_sufficient-qc1}]
The proof is based on embedding of $h$ into a Loewner chain associated with a Herglotz function not depending on~$t$.
For a suitable $\omega\in\Complex^*:=\Complex\setminus\{0\}$, we have $B=\omega U(k)$. It follows that the function~$p(z,t):=\omega/h'(z)$, $z\in\UH$, $t\ge0$, takes values in~$U(k)$. In particular, $p$ is a Herglotz function in~$\UH$. Denote by $(\varphi_{s,t})$ its associated evolution family.

Set $h_t:=h-\omega t$ for all $t\ge0$. Clearly, $(d/dt)h_t(\varphi_{s,t}(z))=0$ for all $z\in\UH$ and all~$t\ge s\ge0$. Therefore, $h_t\circ\varphi_{s,t}=h_s$ whenever $t\ge s\ge0$. Moreover, since $\Re h'/\omega>0$, by the Noshiro\,--\,Warschawski Theorem (see, e.g., \cite[p.\,47]{Duren:book}), $h_t$'s are univalent in~$\UH$. Thus by \cite[Lemma~3.2]{MR2789373}, $(h_t)$ is a Loewner chain associated with~$(\varphi_{s,t})$, and the desired statement follows now directly from Theorem~\ref{TH_alldetails}.
\end{proof}

It is known, see \cite{Krzyz::StronglyStarlike}, that if the inequality
\begin{equation}\label{EQ_StrongStarlike}
\left|\arg\frac{zf'(z)}{f(z)}\right|\le\arcsin k
\end{equation}
holds for all $z\in\UD\setminus\{0\}$ and some~$k\in[0,1)$, then  $f$ is $k$-q.c. extendible; see also~\cite{Sugawa} for a generalization of this result.
Using Becker's Theorem~\ref{TH_Beckerthm}, one can prove  that the condition ${zf'(z)/f(z)\in U(k)}$ for all $z\in\UD\setminus\{0\}$ also implies $k$-q.c. extendibility.\footnote{See \cite[Theorem 1]{Brown:1984} and \cite{SugawaPol} for other proofs.} Two different generalizations of the latter criterion has been recently established by the second author~\cite[Proposition~2, Theorem~3]{Hotta::LeadingCoefficient}. Analyzing the arguments in the proof of Theorem~\ref{TH_sufficient-qc1} we are able to obtain a criterion improving both of these generalizations.

\begin{corollary}\label{C_qc-sufficient}
Let $k\in[0,1)$ and ${K:=(1+k)/(1-k)}$. Let ${f\in\Hol(\UD,\Complex)}$. Suppose that $f(0)=0$ and $f'(0)\neq0$. If there exists a closed disk~$B\subset\Complex^*$  such that $\max\limits_{w,z\in B}\sqrt{|w/z|}\le K$ and
\begin{equation}\label{EQ_inCorollary}
\frac{zf'(z)}{f(z)}\in B\quad\text{for all $z\in\UD\setminus\{0\}$},
\end{equation}
then $f$ has a $k$-q.c. extension to~$\ComplexE$ with a fixed point at~$\infty$.
\end{corollary}
\begin{proof}
Choose any single-valued branch of $h(z):=-\log f(e^{-z})$ in~$\UH$. Then by Theorem~\ref{TH_sufficient-qc1}, $h$ has a $k$-q.c. extension~$\tilde h$ to~$\ComplexE$ with a fixed point at~$\infty$.  Moreover, since $f$ has a simple zero at the origin, $h$ commutes with the shift map $w\mapsto w+2\pi i$. To obtain a $k$-q.c. extension~$\tilde f$ of~$f$ to~$\C$, now it is sufficient to use the fact that $\tilde h$, see Remark~\ref{RM_explicit} for the explicit formula, inherits this property, i.e., $\tilde h(w+2\pi i)=\tilde h(w)+2\pi i$ for all $w\in\Complex$. Finally, since  $\tilde f(\Complex)\subset\Complex$ by construction, $\tilde f$ extends quasiconformally to~$\ComplexE$ by setting $\tilde f(\infty):=\infty$.
\end{proof}

\begin{remark}
It is easy to see from the above proof that if we suppose  in Corollary~\ref{C_qc-sufficient} that instead of a simple zero, $f$ has a multiple zero at~$0$, then $f$ would have an extension up to a $k$-quasiregular ramified covering of~$\ComplexE$ with ramification points at~$0$ and~$\infty$.
\end{remark}

\REM{By the \textit{hyperbolic radius} and \textit{center} of a disk $D\subset\UH$ we mean the radius and center of~$D$ regarded as a ball w.r.t. the hyperbolic distance in~$\UH$, $$\dist_{hyp,\UH}(z_1,z_2):=\frac12\log\frac{1+\rho_\UH(z_1,z_2)}{1-\rho_\UH(z_1,z_2)},\quad\text{ where } \rho_\UH(z_1,z_2):=\left|\frac{z_1-z_2}{z_1+\overline z_2}\right| \text{ for all $z_1,z_2\in\UH$}.$$}

Note that $U(k)$ is the hyperbolic disk in~$\UH$ of  radius~$\frac12\log(1+k)/(1-k)$ and  centered at~$1$. In the proof of the following theorem we make use of Corollary~\ref{C_qc-ext}. This allows us to replace\footnote{The same can be done in the context of Theorem~\ref{TH_sufficient-qc1}. However, it would not lead to any improvement.} the disk $U(k)$ with a disk~$D$ of the same hyperbolic radius but centered at an arbitrary point of~$\UH$.
\begin{theorem}\label{TH_sufficient-qc2}
Let $k\in[0,1)$ and ${K:=(1+k)/(1-k)}$. Let ${h:\UH\to\Complex}$ be a holomorphic function. Suppose that there exists a closed hyperbolic disk~$D\subset\UH$ of radius $\frac12\log K$ such that the condition
\begin{equation}\label{EQ_sufficiet_qc2}
\dfrac{\vphantom{\int\nolimits^1}h(z)+a}{h'(z)}-z\,\in\, D
\end{equation}
holds for all $z\in\UH$ and some~$a\in\Complex$.
Then $h$ has a $k$-q.c. extension to~$\ComplexE$ with a fixed point at~$\infty$.
\end{theorem}
\begin{remark}
In fact, we can say more about  holomorphic functions satisfying conditions of the above theorem, see Lemma~\ref{LM_f}.
\end{remark}
\begin{example}
Let $h(z):=\sqrt{(z+1)^2+\alpha}$, with some $\alpha\in\Complex$. Then $h$ is holomorphic in~$\UH$ if and only if $|\alpha|\le2+\Re\alpha$. Moreover, if the latter inequality is strict, then $h$ satisfies the hypothesis of the above theorem with~$K:=\sqrt{(2+\Re\alpha+|\alpha|)/(2+\Re\alpha-|\alpha|)}$ and~${a=0}$.
\end{example}
\begin{remark}
The hypotheses in Theorems~\ref{TH_sufficient-qc1} and~\ref{TH_sufficient-qc2} imply certain geometric properties of $h(\UH)$. Namely, if $h$ satisfies the hypothesis of Theorem~\ref{TH_sufficient-qc1}, then $h(\UD)$ is convex in one direction, i.e. $w+e^{i\theta}t\in h(\UD)$ for all~$w\in h(\UD)$, all $t\ge0$ and some~$\theta\in\Real$ independent of~$w$ and~$t$, see, e.g., \cite[Chapter\,3.5]{ElinShoikhet:book}. Similarly, condition~\eqref{EQ_sufficiet_qc2} implies that~$h(\UD)$ is starlike w.r.t.~$\infty$, i.e. $e^tw\in h(\UH)$  for all~$w\in h(\UD)$ and all $t\ge0$, see, e.g.,  \cite[Chapter\,3.1]{ElinShoikhet:book}. Therefore, although~\eqref{EQ_sufficiet_qc2} implies that both $h'$ and $1/h'$ are bounded in~$\UH$, it is easy to find an example\REM{\footnote{Take any conformal map~$h_0:\UH\Mapinto\UH$ having a simple pole at~$\infty$ and such that~$h(\UD)$ is starlike w.r.t.~$\infty$ but not convex in any direction, e.g., $h_0(z):=2\pi i\,\big/\left(\frac{i\pi}2+\log\frac{z+i}{iz+1}\right)$, which maps $\UH$ onto ${\{w\in\UH\colon|w-1|>1\}}$. Then for all $\varepsilon>0$ small enough, $h_\varepsilon(z):=h_0(z+\varepsilon)$ satisfies the hypothesis of Theorem~\ref{TH_sufficient-qc2} with some $k=k(\varepsilon)\in(0,1)$ but not that of Theorem~\ref{TH_sufficient-qc1}.}} of a function~$h$ that satisfies~\eqref{EQ_sufficiet_qc2} for some~$k\in(0,1)$ but does not satisfy the hypothesis of Theorem~\ref{TH_sufficient-qc1} for any~$k\in(0,1)$.
\end{remark}

\begin{proof}[\fontseries{bx}\selectfont Proof of Theorem~\ref{TH_sufficient-qc2}]
We will follow the proof of Theorem~\ref{TH_sufficient-qc1}, which implicitly uses one-parameter semigroups of holomorphic self-maps of~$\UH$. In the proof of the present theorem, one-parameter semigroups play more substantial role. Readers unfamiliar with the basic theory of such semigroups are  referred to, e.g., \cite{Abate:book, ElinShoikhet:book,Shoikhet:2001}.

 First of all, we obviously may assume~$a=0$. Then $p(z,t):=h(z)/h'(z)$ is a Herglotz function, with $p'(\infty,t)=1$ for all~$t\ge0$. Consider the evolution family~$(\varphi_{s,t})$ associated with~$p$. Set $h_t:=e^{-t}h$ for all~$t\ge0$. Then $h_t\circ\varphi_{s,t}=h_s$ whenever $t\ge s\ge0$. Note also that since $p(z,t)$ does not depend on~$t$, we have $\varphi_{s,t}=\varphi_{0,t-s}=:\phi_{t-s}$ whenever $t\ge s\ge0$. The family~$(\phi_t)_{t\ge0}$ is a one-parameter semigroup in~$\UH$ with the infinitesimal generator~$G(z):=h(z)/h'(z)$ and the Denjoy\,--\,Wolff point at~$\infty$. Condition~\eqref{EQ_sufficiet_qc2} implies that $h$ does not vanish in~$\UH$ and hence, choosing any branch of~$\log$,  the function ${\sigma:=\log h}$ is well-defined and single-valued in~$\UH$. Moreover, $\sigma\circ\phi_t=\sigma+t$ for all~$t\ge0$, and $\Re \sigma'=\Re(1/p)>0$, from which, by the Noshiro\,--\,Warschawski Theorem, it follows that $\sigma$ is univalent in~$\UH$. Note also that $G'(\infty)=1$ due to condition~\eqref{EQ_sufficiet_qc2}. Then combining \cite[Theorem\,\,2.1]{ContrMadrigal:AnaFlows} with \cite[Theorem\,1]{ContrMadrigalPomm:2006}, see also~\cite[Theorem\,7]{Elin-Khavinson-Reich-Shoikhet:2010}, and taking into account the ``absorption property'' (see, e.g., \cite[Remark~4.12]{Gumenyuk:2014}), we conclude that there exists~$c\in\Real$ such that:
\begin{itemize}
\item[(a)] $\sigma(\UH)\subset\mathcal D:=\{w\colon c-\pi/2<\Im w<c+\pi/2\}$, and\\[-1.5ex]
\item[(b)] $\bigcup_{t\ge0}[\sigma(\UH)-t]=\mathcal D$.
\end{itemize}
Taking into account~(a), the univalence of $\sigma$ implies the univalence of all $h_t$'s, which in turn implies that $(h_t)$ is a Loewner chain associated with~$(\varphi_{s,t})$. Furthermore, by~(b) the Loewner range of~$(h_t)$ is a half-plane. Condition~\eqref{EQ_sufficiet_qc2} implies that the Herglotz function~$p$ satisfies~\eqref{EQ_U(k)-unbdd}. Therefore it follows, due to Corollary~\ref{C_qc-ext}, that $h=h_0$ is $k$-q.c. extendible.
It remains to notice that condition~\eqref{EQ_sufficiet_qc2} implies that $\int_1^{+\infty}\!\big(h'(x)/h(x)\big)\,\mathrm{d}x=+\infty$ and hence the extension of~$h$ has to fix~$\infty$.
\end{proof}
\REM{\begin{remark}
\nv{It is worth to mention that in contrast to Becker's classical construction, the range of the Loewner chain~$(h_t)$ coming into play in the above proof does not cover the whole complex plain. Therefore, analyzing the proof of Corollary~\ref{C_qc-ext}, we see that the conclusion on the $k$-q.c. extendibility is obtained by using the explicit q.c.-extension for the evolution family~$(\varphi_{s,t})$ rather than that for Loewner chains.}
\end{remark}}

\subsection{Another q.c.-extendibility condition in terms of the first derivative}
Our next aim is to prove the following, analogous, to some extent, to~\cite[Theorem~3]{Hotta::Explicit}.
\begin{theorem}\label{TH_QCvia-derivative}
Fix $k\in[0,1)$ and let $h\colon\UH\to\Complex$ be a holomorphic function. Suppose there exists a closed hyperbolic disk $D\subset\UH$ of radius~$\frac12\log\frac{1+k}{1-k}$ and a holomorphic  function $f\colon\UH\to\Complex$ such that the following
two statements hold:
\begin{itemize}
\item[(A)] ${f(z)}/{f'(z)}-z\in D$ for all~$z\in\UH$, and
\item[(B)] $\big(h'(z)f(z)\big)^{-1}-z\in D$ for all~$z\in\UH$.
\end{itemize}
Then $h$ is univalent in~$\UH$ and has a $k$-q.c. extension to~$\ComplexE$.
\end{theorem}
\begin{remark}
The q.c.-extension of~$h$ constructed in the proof of the above theorem has a finite value at~$\infty$.
\end{remark}

Before giving a proof of the above theorem we deduce a curious consequence.
\begin{corollary}\label{CR_QCvia-derivative}
Let $\psi:\UD\to\Complex$ be a holomorphic function. If the inequality
\begin{equation}\label{EQ_varphi-bound-k}
\mathrlap{\left|\frac{1-\psi'(\zeta)}{1+\zeta\psi'(\zeta)}\right|\le k<1}
\hphantom{\left|\frac{1-\psi'(z)}{1+z\psi'(z)}\right|\le 1\quad\text{for all~$z\in\UD$},}
\end{equation}
holds for all~$\zeta\in\UD$ and some~$k\in[0,1)$, then $\psi$ admits a $k$-q.c. extension to~$\ComplexE$.
\end{corollary}
\begin{proof} It is sufficient to apply Theorem~\ref{TH_QCvia-derivative} to $D:=U(k)$, $f(z):=z+1$, and ${h(z):=-\frac12\psi(\frac{1-z}{1+z})}$ for all~$z\in\UH$.
%\dv{$h:=\psi\circ H^{-1}$, where $H(\zeta):=(1+\zeta)/(1-\zeta)$ is the Cayley map of~$\UD$ onto~$\UH$.}
\end{proof}

\begin{remark}
From the above corollary it follows that the inequality
\begin{equation}\label{EQ_varphi-bound-1}
\left|\frac{1-\psi'(z)}{1+z\psi'(z)}\right|\le 1\quad\text{for all~$z\in\UD$}
\end{equation}
is a sufficient condition for a non-constant holomorphic function~$\psi$ in~$\UD$ to be univalent.
Indeed, if $\psi$ satisfies~\eqref{EQ_varphi-bound-1}, then it is the locally uniform limit, as ${k\to1^-}$,  of the functions $$\psi_k(z):=\psi(0)+\int_0^z\dfrac{1-k\varphi(\zeta)}{1+k\zeta\varphi(\zeta)}d\zeta,\quad z\in\UD,\quad\varphi(\zeta):=\frac{1-\psi'(z)}{1+z\psi'(z)},$$
satisfying~\eqref{EQ_varphi-bound-k}.

However, this does not constitute a new univalence criterion. Rewriting~\eqref{EQ_varphi-bound-1} in the form $|1-\psi'(z)|^2\le|1+z\psi'(z)|^2$, we see that it implies the inequality $2\,\Re\big(\psi'(z)/g'(z)\big)\ge(1-|z|^2)|\psi'(z)|^2$, where $g(z):=\log(1+z)$ is a convex function, from which it follows that $\psi$ is close-to-convex, see, e.g., \cite[\S2.6]{Duren:book}. Moreover, if $\psi'$ is continuous up to the boundary and $1+z\psi'(z)$ does not vanish in~$\UD$, then the condition ${\Re\big(\psi'(z)/g'(z)\big)>0}$ for all $z\in\UD$ implies~\eqref{EQ_varphi-bound-1}.
Therefore, one can regard Corollary~\ref{CR_QCvia-derivative} as a ``q.c.-version'' of the close-to-convexity with a particular choice of the function~$g$. In this connection, it would be interesting to see whether a similar relation can be found between Theorem~\ref{TH_QCvia-derivative} and a suitable analogue of close-to-convexity for the half-plane.
\end{remark}

One of the ``ingredients'' in the proof of Theorem~\ref{TH_QCvia-derivative} is the following modification of~Theorem~\ref{Th_suff-for-chordal-Loewner-chain}.

\begin{proposition}
\label{PR_LoewChain}
Let $(h_{t})_{t \ge 0}$ be a family of holomorphic functions in $\UH$ satisfying the following conditions:
\begin{enumerate}
\item[(i)] the map $\UH\times[0,+\infty)\ni (z,t)\mapsto h_t(z)$ is continuous and for each~$z\in\UH$ the map $[0,+\infty)\ni t\mapsto h_t(z)$ is locally absolutely continuous;
\item[(ii)] there exists a constant $C_1>0$ and a Herglotz function $p :\UH\times[0,+\infty) \to \{w\colon\Re w> C_1\}$ such that for each~$z\in\UH$,
\begin{equation}\label{EQ_LoewPDE_a.e.}
\frac{\partial h_t(z)}{\partial t}=-h_t'(z)p(z,t)\quad \text{for a.e. $t\ge0$};
\end{equation}
\item[(iii)] there exists $a>0$ such that $$\displaystyle \inf_{\Re z > a + C_{1}t} \mathsf{R}_u^{\mathsf{hyp}}(h_{t}, z) \to \infty\quad\text{ as~~~$t \to +\infty$,}$$ where $\mathsf{R}_u^{\mathsf{hyp}}(f,z)$ stands for the radius of the largest hyperbolic disk in $\UH$ centered at $z$ in which $f$ is univalent.
\end{enumerate}
Then, $(h_{t})_{t \ge 0}$ is a chordal Loewner chain in $\UH$ associated with the Herglotz function $p$.
\end{proposition}
\begin{proof}
Denote by~$(\varphi_{s,t})$ the evolution family associated with the Herglotz function~$p$.
As in the proof of Theorem~\ref{Th_suff-for-chordal-Loewner-chain},  we see that
\begin{equation}\label{EQ_hsht}
h_s=h_t\circ\varphi_{s,t}\quad \text{ for any $s\ge0$ and any $t\ge s$.}
\end{equation}
Using this equality, we will show that condition~(iii) implies that $h_s$ is univalent in~$\UH$ for all~$s\ge0$, which in its turn implies the conclusion of Proposition~\ref{PR_LoewChain}. To this end, fix any~$s\ge0$ and choose a point~$z_0\in\UH$ with $\Re z_0>a+C_1 s$.

Let $D$ be any hyperbolic disk in~$\UH$ centered at~$z_0$. Denote $z(t):=\varphi_{s,t}(z_0)$.
Since ${\Re p\ge C_1}$, we have $\Re z(t)\ge a+C_1 t$ for all $t\ge s$. By the Schwarz\,--\,Pick Lemma, for every $t\ge s$, $\varphi_{s,t}(D)$ lies in the hyperbolic disk $D(t)$ centered at~$z(t)$ and having the same radius as the disk $D$.
By condition~(iii), $h_t$ is univalent in~$D(t)$ provided $t\ge s$ is large enough. From~\eqref{EQ_hsht} it then follows that $h_s$ is univalent in~$D$. Since the radius of $D$ can be chosen arbitrarily large, this shows that $h_s$ is univalent in~$\UH$ and hence the proof is complete.
\end{proof}

\subsection{Lemmas}
The following three lemmas will be used in the proof of Theorem~\ref{TH_QCvia-derivative}.
\begin{lemma}\label{LM_f}
Let $f\colon\UH\to\Complex$ be a holomorphic function satisfying condition~(A) in Theorem~\ref{TH_QCvia-derivative}. Then the following assertions hold:
\begin{enumerate}
\item[(a)] $\lambda f(\UH) \supset f(\UH)$ for any $\lambda \in (0,1]$;
\item[(b)] $\vphantom{\displaystyle{\int_1^1}}\bigcup_{\lambda \in (0,1]} \lambda f(\UH) = e^{i\theta} \UH$ for some $\theta \in \R$;
\item[(c)] the limit $f'(\infty):=\displaystyle\vphantom{\displaystyle\int\limits_1}\lim_{\UH \ni z \to \infty}{f(z)}/{z}$ exists finitely, with $e^{-i\theta}f'(\infty) > 0$;
\item[(d)] $f'(z) \to f'(\infty)$ as $\UH \ni z \to \infty$;
\item[(e)] $f$ and $1/f$ are bounded on every bounded subset of~$\UH$.
\end{enumerate}
\end{lemma}
\begin{proof}
According to~(A), $f$ satisfies condition~(ii) of Theorem~\ref{TH_sufficient-qc2} with~$a=0$. Assertions~(a) and~(b) follow directly from the proof of Theorem~\ref{TH_sufficient-qc2}.

Note also that by~(A), $f$ does not vanish and hence $g(z):=\log \big(f(z)/z\big)$ is well defined in~$\UH$ (The choice of the branch is not important). Moreover, condition (A) implies that $|g'(z)|<M/|z|^2$ for all $z\in\UH$ and some constant~$M>0$. Therefore, $g(z)$ has a finite limit as $\UH\ni z\to\infty$, which implies the existence of a non-vanishing limit in~(c). Furthermore, again by~(A), $e^{g(z)}/f'(z)=\big(f(z)/f'(z)\big)/z\to1$ as $\UH\ni z\to\infty$, This proves~(d).

To see that ${e^{-i\theta}f'(\infty)>0}$ it is sufficient now to apply the Julia\,--\,Wolff\,--\,Carath\'eodory Theorem for the half-plane, see, e.g.~\cite[Ch.\,IV~\S26]{Valiron}, to the function $e^{-i\theta}f$. This completes the proof of~(c).

Now consider the function~$q(z):=\log f(z)$, $z\in\UH$. From~(A) it follows that $$\Re \big(q\,'(z)\big)^{-1}>\min\{ \Re \zeta\colon\zeta\in D\}>0\quad \text{for all~$z\in\UD$.}$$ Hence $q\,'$ is bounded in~$\UH$. This easily implies assertion~(e).
\end{proof}

\begin{lemma}\label{LM_h_t-image-union}
Under the conditions of Theorem~\ref{TH_QCvia-derivative},
$$
\Omega:=\bigcup_{t \ge 0} h_{t}(\UH) = w_0-e^{-i\theta}\UH,
$$
where $w_0\in\Complex$, $h_t:=h-(e^t-1)/f$ for all~$t\ge0$, and $\theta\in\Real$ is the same as in Lemma~\ref{LM_f}.
\end{lemma}
\begin{proof}
First of all, as we mentioned in the proof of Lemma~\ref{LM_f}, $f$ does not vanish in~$\UH$. Hence $h_t$'s are holomorphic there.

Considering $-e^{-i\theta} f$ and $-e^{i\theta} h$ instead of $f$ and $h$, respectively, we may assume that~${\theta=\pi}$. Then by assertion~(b) of Lemma~\ref{LM_f}, $\Re f<0$. Therefore, $h_t(\UH)\subset u+\UH$ unless $u:=\inf_{z\in\UH}\Re h(z)=-\infty$.

To see that $u$ is finite, we will show that $h$ has finite limits everywhere on~$\partial\UH$. Consider first a \emph{finite} point~$z_0\in\partial\UH$. Combining condition~(B) with assertion~(e) of Lemma~\ref{LM_f} applied to the bounded set~$E:=\{z\in\UH\colon|z-z_0|<1\}$, we conclude that $h'$ is bounded in~$E$. Hence~$h$ has a finite limit at~$z_0$.

Similarly, from condition~(B) and assertion~(c) of Lemma~\ref{LM_f} it follows easily that
\begin{equation}\label{EQ_zh}
z^2h'(z)\to1/f'(\infty)\quad\text{as~~~~$\UH\ni z\to\infty$},
\end{equation}
which implies that also at~$z_0:=\infty$, $h$ has finite limit, which we will denote by~$h(\infty)$.

Observe now that $u =\Re h(\infty)$. Indeed, recall that we have assumed $\theta=\pi$. Assertion~(b) of Lemma~\ref{LM_f} combined with condition~(B) implies that $h'(\UH)\subset\Complex\setminus [0,+\infty)$. Fix any $x>0$. If  $\Re h|_{\{z\colon \Re z\ge x\}}$ attains minimum at some finite point $z_*$, then $\Re z_*=x$ and $h'(z_*)>0$, which contradicts the previous conclusion. Since $x>0$ can be chosen arbitrarily small, this means that $\Re h(z)>\Re h(\infty)$ for all~$z\in\UH$.

Thus we have proved that $\Omega\subset h(\infty)+\UH$.  It remains to show that $\Omega\supset h(\infty)+\UH$. To this end we fix an arbitrary $w$ with $\Re w>u$ and prove that the equation $h_t(z)=w$ has at least one solution in~$\UH$ provided $t$ is large enough. Assuming that $t>0$, rewrite the equation $h_t(z)=w$ in the following form
\begin{equation}\label{EQ_equation-for-Rouche}
\frac{f(z)}{e^t-1}-\frac{1}{h(\infty)-w}+Q(z)=0,~~~\text{where~~}Q(z):=\frac{h(z)-h(\infty)} {(h(\infty)-w)(h(z)-w)}.
\end{equation}
By assertion~(b) of Lemma~\ref{LM_f}, the equation $P(z):=f(z)/(e^t-1)-1/(h(\infty)-w)=0$
 has a solution $z(t)\in\UH$ for all $t>0$ large enough, say $t>t_0(w)$. By assertions (c) and~(e) of the same lemma,
\begin{equation}\label{EQ_z(t)}
e^{-t}z(t)\to\omega\quad\text{as~~~$~~~t\to+\infty$}\qquad\text{for some~~$\omega\in\UH$.}
\end{equation}
We will apply Rouch\'e's Theorem to show that~\eqref{EQ_equation-for-Rouche} has also a solution in~$\UH$.
Consider the disk $W(t):=\{z\colon |z-z(t)|<\Re z(t)/2\}\subset\subset\UH$. Since $\inf\{|z|\colon z\in W\}\to+\infty$ as $t\to+\infty$ and $\Re h(z)\to u>\Re w$ as $\UH\ni z\to\infty$, the function $Q$ is holomorphic in~$W(t)$ for all $t>0$ large enough, say $t>t_1(w)\ge t_0(w)$.
Recall that $f$ is univalent in~$\UH$ by Theorem~\ref{TH_sufficient-qc2}. Therefore, by the Koebe one-quarter theorem,
\begin{equation*}\label{EQ_est-for-P}
|P(z)|>\frac{|f'(z(t))|}{4e^t}\cdot\frac{\Re z(t)}2\quad\text{for all}~~~z\in\partial W(t)
\end{equation*}
and all~$t>t_0(w)$, while from~\eqref{EQ_zh} it follows that
$$
-zf'(\infty)\big(h(z)-h(\infty)\big)=\int\limits_1^{+\infty}z^2f'(\infty)h'(\tau z)d\tau\to\int\limits_1^{+\infty}\frac{d\tau}{\tau^2}=\pi/2\quad\text{ as }~~~\UH\ni z\to\infty.
$$
Therefore, with the help of~\eqref{EQ_z(t)} and assertion~(d) of Lemma~\ref{LM_f}, we conclude that $|P(z)|>|Q(z)|$ for all $z\in\partial W(t)$ provided $t>t_1(w)$ is large enough. Then by Rouch\'e's Theorem, for all such~$t$'s,~\eqref{EQ_equation-for-Rouche} has a solution in~$W(t)$. This completes the proof.
\end{proof}

\begin{lemma}\label{LM_h_t-univalence}
Under the conditions of Lemma~\ref{LM_h_t-image-union},
for any $R > 0$, there exists $\tau>0$ and $a>0$ such that for all $t>\tau$ and all $z\in\UH(a)$, $h_{t}$ is univalent in the hyperbolic disk of radius $R$ centered at $z$.
\end{lemma}
\begin{proof}
First of all, note that from~\eqref{EQ_zh} and assertions~(c) and~(d) of Lemma~\ref{LM_f} it follows that there exists $a_0>0$ such that $z^2h'_t(z)\neq0$ for all $z\in\UH(a_0)$  and all~$t\ge0$.

Given $z\in\UH(a_0)$ and $t\ge0$, we define $$F_{z,t}(\zeta):=\frac{h_t\big(\zeta\,\Re z+i\,\Im z\big)-h_t(z)}{h_t'(z)\,\Re z},\quad\zeta\in\UH.$$

Since hyperbolic disks in~$\UH$ are invariant w.r.t. automorphisms of~$\UH$, it is sufficient to show that given a hyperbolic disk $D$ centered at~$1$, $F_{z,t}$ is univalent on~$D$ provided $\Re z$ and $t>0$ are large enough. Suppose that this statement does not hold. Then there exist sequences $(z_n)\subset\UH(a_0)$ and $(t_n)\subset[0,+\infty)$ with $\lim_{n\to+\infty}\Re z_n=+\infty$ and $\lim_{n\to+\infty}t_n=+\infty$ such that for any~$n\in\Natural$ the map $F_{z_n,t_n}$ is not injective on~$D$.

Fix any closed hyperbolic disk~$K$ containing $D$ in the interior. Denote $g:=1/f$. Recall that $f$, and hence $g$, are univalent in~$\UH$ by Theorem~\ref{TH_sufficient-qc2}. Therefore, the functions
\begin{equation*}
U_n(\zeta):=\frac{h\big(\zeta\,\Re z_n+i\,\Im z_n\big)-h(z_n)}{g\big(\zeta\,\Re z_n+i\,\Im z_n\big)-g(z_n)}
\end{equation*}
are all holomorphic in~$\UH$. By assertions (c) and~(d) of Lemma~\ref{LM_f}, $z^2g'(z)\to-1/f'(\infty)$ as $\UH\ni z\to\infty$. Together with~\eqref{EQ_zh}, this implies that
\begin{equation}\label{EQ_U_n-limit}
U_n(\zeta)\to -1
\qquad\text{uniformly on~$K$ as $n\to+\infty$.}
\end{equation}
Taking into account that $h'(z_n)/g'(z_n)=U_n(1)$, from~\eqref{EQ_U_n-limit} we get
\begin{multline}\label{EQ_remove-h}
F_{z_n,t_n}(\zeta)\cdot\left(\frac{f\big(\zeta\,\Re z_n+i\,\Im z_n\big)-f(z_n)}{f'(z_n)\,\Re z_n}\,\frac{f(z_n)}{f\big(\zeta\,\Re z_n+i\,\Im z_n\big)}\right)^{-1}=\\=\quad F_{z_n,t_n}(\zeta)\cdot\left(\frac{g\big(\zeta\,\Re z_n+i\,\Im z_n\big)-g(z_n)}{g'(z_n)\,\Re z_n}\right)^{-1}=\quad\frac{U_n(\zeta)}{U_n(1)}\quad\to\quad1\qquad\hbox{~}
\end{multline}
uniformly on~$K$ as $n\to+\infty$.

Passing if necessary to a subsequence, we may assume that $\Im z_n/|z_n|\to\eta\in[-1,1]$ as $n\to+\infty$. Then, using again assertions (c) and~(d) of Lemma~\ref{LM_f}, from~\eqref{EQ_remove-h} we deduce that
$$
F_{z_n,t_n}(\zeta)\to F_0(\zeta):=(\zeta-1)\frac{\xi+i\eta}{\xi\zeta+i\eta}, \quad \xi:=\sqrt{1-\eta^2},
$$
uniformly on~$K$ as $n\to+\infty$. Since $F_0$ is univalent in~$\UH$, $F_{z_n,\tau_n}$ is univalent in~$D$ for all~$n$ large enough. With this conclusion contradicting our original assumption, the lemma is now proved.
\end{proof}

\subsection{Proof of Theorem~\ref{TH_QCvia-derivative}}
Note again that by~(A), $f$ does not vanish in~$\UH$. Therefore, the functions $h_t:=h-(e^t-1)/f$, $t\ge0$, are holomorphic in~$\UH$.
It is easy to check that $h_t(z)$ satisfies the PDE $\partial h_t(z) /\partial t=-h_t'(z)/q(z,t)$, where $q(z,t):=e^{-t}A(z)+(1-e^{-t})B(z)$, $A(z):=h'(z)f(z)$, and $B(z):=f'(z)/f(z)$ for all $z\in\UH$ and all~$t\ge0$. By the hypothesis of the theorem, $A(z)$ and $B(z)$ belong to~$D_z:=\{w\colon\frac1w-z\in D\}\subset\UH$ for all~$z\in\UH$. Since for any~$z\in\UH$, $D_z$ is a convex set, it follows that $q(z,t)\in D_z$ for all~$z\in\UH$ and all~$t\ge0$. In particular, $p:=1/q$ is a Herglotz function in~$\UH$.

Therefore, applying Corollary~\ref{C_qc-ext} with suitable constant functions~$\alpha$ and~$\beta$, we see that the range-normalized Loewner chains associated with the Herglotz function~$p$ can be extended $k$-quasiconformally to~$\ComplexE$. Hence it remains to show that $(h_t)$ \emph{is} a range-normalized Loewner chain. The first part, i.e. the fact that $h_t$ is a Loewner chain, follows from Lemma~\ref{LM_h_t-univalence} and Proposition~\ref{PR_LoewChain}. The second part,  i.e. the fact that this Loewner chain is range-normalized, is the statement of Lemma~\ref{LM_h_t-image-union}. This completes the proof.\proofbox

    %   ++++++++++++++++++++++++++++++++++++++++++++++++++++++
    %
    %       Reference
    %
    %   ++++++++++++++++++++++++++++++++++++++++++++++++++++++

\vskip2ex\noindent {\bf Acknowledgement.}
The authors are grateful to Professors Dmitri Prokhorov and Toshiyuki Sugawa for interesting discussions concerning Corollary~\ref{CR_QCvia-derivative}.

\bibliographystyle{amsalpha}

\end{document}